\documentclass[11pt,a4paper]{amsproc}

\usepackage[usenames]{color}
\usepackage{amsmath, amssymb, amsthm, latexsym,nicefrac,setspace, todonotes, graphicx, mathrsfs,mathtools}
\usepackage[pagebackref,colorlinks,linkcolor=red,citecolor=blue,urlcolor=blue,hypertexnames=true]{hyperref}
\usepackage{tikz}
\usetikzlibrary{positioning, arrows.meta}

\theoremstyle{plain}
\newtheorem{theorem}{Theorem}[section]
\newtheorem{question}[theorem]{Question}
\newtheorem{lemma}[theorem]{Lemma}
\newtheorem{corollary}[theorem]{Corollary}
\newtheorem{proposition}[theorem]{Proposition}

\theoremstyle{definition}
\newtheorem{definition}[theorem]{Definition}
 
\theoremstyle{remark}
\newtheorem{remark}[theorem]{Remark}

\DeclareMathOperator{\Iso}{Iso}
\DeclareMathOperator{\Neigh}{\mathcal{U}}
\DeclareMathOperator{\Samuel}{S}
\DeclareMathOperator{\Mean}{M}
\DeclareMathOperator{\Meanu}{M_{\it u}}
\DeclareMathOperator{\R}{\mathbb{R}}
\DeclareMathOperator{\Pfin}{\mathcal{P}_{fin}}
\DeclareMathOperator{\LUEB}{LUEB}
\DeclareMathOperator{\LUCB}{LUCB}
\DeclareMathOperator{\RUEB}{RUEB}
\DeclareMathOperator{\RUCB}{RUCB}
\DeclareMathOperator{\UEB}{UEB}
\DeclareMathOperator{\UCB}{UCB}
\DeclareMathOperator{\UMF}{UMF}
\DeclareMathOperator{\UU}{U}
\DeclareMathOperator{\spt}{spt}
\DeclareMathOperator{\diam}{diam}
\DeclareMathOperator{\Lip}{Lip}
\DeclareMathOperator{\Prob}{P}

\newcommand{\norm}[1]{\left\Vert #1 \right\Vert}

\newcommand{\defeq}{\mathrel{\mathop:}=}                         
 
\newcommand{\mc}{\mathcal}
\newcommand{\lacts}{\curvearrowright}

\def\dim{{\rm dim}}

\def\R{\mathbb{R}}

\def\N{\mathbb{N}}

\newcommand{\eps}{\varepsilon}
\newcommand{\eqa}[1]{
\begin{align*}
#1
\end{align*}
}

\newcommand{\ip}[2]{\langle #1,#2\rangle}

\author[V. Alekseev]{Vadim Alekseev}
\address{V.A., TU Dresden, 01062 Dresden, Germany}
\email{vadim.alekseev@tu-dresden.de}
\author[H. Ando]{Hiroshi Ando}
\address{H.A., Chiba University, 1-33 Yayoi-cho, Inage, Chiba, 263-8522,
Japan}
\email{hiroando@math.s.chiba-u.ac.jp}
\author[F.M. Schneider]{Friedrich Martin Schneider}
\address{F.M.S., TU Bergakademie Freiberg, Prüferstraße 1, 09596 Freiberg, Germany}
\email{martin.schneider@math.tu-freiberg.de}
\author[A. Thom]{Andreas Thom}
\address{A.T., TU Dresden, 01062 Dresden, Germany}
\email{andreas.thom@tu-dresden.de}

\title[Amenability and skew-amenability of actions]{Amenability and skew-amenability of actions of topological groups}

\begin{document}

\begin{abstract} 
We define and study notions of amenability and skew-amenability of continuous actions of topological groups on compact topological spaces. Our main motivation is the question under what conditions amenability of a topological group passes to a closed subgroup. Other applications include the understanding of the universal minimal flow of various non-amenable groups.
\end{abstract}

\dedicatory{Dedicated to the memory of Anatoly Vershik.}

\maketitle

\tableofcontents

\section{Introduction}

The notion of amenability of topological groups beyond the locally compact setting has some surprises in store (see~\cite{PestovBook,GrigorchukDeLaHarpe} for comprehensive accounts). A topological group $G$ is called \emph{amenable} if there exists a left-invariant mean on the algebra $\RUCB(G)$ of bounded right-uniformly continuous functions on $G$. The Gel'fand spectrum of the C$^*$-algebra $\RUCB(G)$ is denoted by $\Samuel(G)$ and called the Samuel compactification of $G$. It has been observed from the start that an equally interesting definition would ask for a left-invariant mean on the algebra $\LUCB(G)$ of bounded left-uniformly continuous functions. This notion is called \emph{skew-amenability} of $G$ and has been studied by various authors \cite{Greenleaf, JuschenkoSchneider, MR4206535}. Both conditions are known to be equivalent if the group $G$ is locally compact by work of Greenleaf \cite{Greenleaf} or SIN by definition. If $G$ is not locally compact, these conditions are different and it is now known that neither implies the other, see \cite{ozawa} for the first example of a skew-amenable group that is not amenable. 

Many properties, familiar from the locally compact setting, of the class of (skew-)amenable groups fail to hold in general: for example, those classes are no longer closed under taking closed subgroups. This fails even for SIN groups, see \cite{alekseev2023amenability, MR3795479}. In this context, it is a longstanding open problem to decide whether closed subgroups of $\UU(R)$, $R$ the hyperfinite II$_1$-factor, are amenable -- or more specifically, whether $\UU(R)$ admits a discrete free subgroup.

In order to identify natural conditions on a subgroup $H \leq G$ that imply that amenability passes from $G$ to $H$, we are naturally led to a new definition of \emph{amenability} and \emph{skew-amenability of actions} (see Section~\ref{sec:amenable}) of general topological groups on compact spaces. The aim of the present work is to provide a systematic exploration of these notions. These concepts extend the notions of amenability of actions in the discrete case, but the locally compact case needs some careful comparison with the existing definition.  We show (see Section~\ref{sec:comparison}) that an action of a second-countable locally compact group is amenable in the sense of Anantharaman-Delaroche \cite{MR1926869Delaroche} if and only if it is amenable or skew-amenable in our sense, thus all reasonable notions agree is the familiar setting of second-countable locally groups and their actions. This also covers the classical result of Greenleaf \cite{Greenleaf} on equivalence of amenability and skew-amenability for the group itself.

A related notion is the notion of exactness of a locally compact group as introduced in the work of Kirchberg--Wassermann~\cite{KW95,KW99-2, KW99-1} for locally compact groups, see also the work of Brodzki--Cave--Li~\cite{MR3635811}. We define a topological group to be (skew-)exact if it admits an (skew-)amenable continuous action on a compact space.
Following Monod, we say that $H\leq G$ is \emph{amenably embedded} if the action of $H$ on $\Samuel(G)$ is amenable. We call $H \leq G$ \emph{well-placed} if there exists an $H$-left-equivariant, right-uniformly continuous map $\mu \colon G \to \Meanu(H)$, extending the inclusion of $H$, where $\Meanu(H)$ is the space of \emph{uniform measures} on $H$ (see Section~\ref{section:preliminaries} for details). 
One of our main results (see Section \ref{sec:emb}) is that if $H\leq G$ is either amenably embedded or well-placed, then amenability of $G$ implies amenability of $H$. Some of these notions and results have skew versions of equal importance.

Our second main application (see Section \ref{sec:UMF}) of the notion of amenability of actions is to the study of the universal minimal flow of particular topological groups. We show that if the universal minimal flow $\UMF(G)$ of a Polish group $G$ is metrizable, then the action $G \curvearrowright \UMF(G)$ is amenable. This is particularly interesting in cases where $G$ is non-amenable.

We conclude the paper in the last section with further examples and applications.

\section{Uniform spaces and the UEB topology}\label{section:preliminaries}

In this preliminary section we compile some relevant background concerning function spaces associated with uniform spaces. Throughout the entire paper, compact will always include being Hausdorff.

For a start, we briefly clarify some general notation. Let $X$ be a set. We let $\mathcal{P}(X)$ denote the power set of $X$ and we let $\Pfin(X)$ denote the set of all finite subsets of $X$. For any map $\mu \in \R^{X}$, we let $\spt(\mu) \defeq \mu^{-1}(\R\setminus \{ 0 \})$. We consider \begin{displaymath}
	\R[X] \, \defeq \, \left. \! \left\{ \mu \in \R^{X} \right\vert \, \spt(\mu) \text{ finite} \right\} \, \subseteq \, \ell^{1}(X)
\end{displaymath} and \begin{displaymath}
	\Delta (X) \, \defeq \, \{ \mu \in \R[X] \mid \mu \geq 0, \, \Vert \mu \Vert_{1} = 1 \} . 
\end{displaymath} For $\mu \in \ell^{1}(X)$ and $f \in \ell^{\infty}(X)$, let \begin{displaymath}
	\langle \mu,f \rangle \, \defeq \, \sum\nolimits_{x \in X} \mu(x)f(x) .
\end{displaymath} If $G$ is a group, then we define \begin{align*}
	\lambda_{g} &\colon \, G \, \longrightarrow \, G, \quad x \, \longmapsto \, gx , \\
    \rho_{g} &\colon \, G \, \longrightarrow \, G, \quad x \, \longmapsto \, xg 
\end{align*} for every $g \in G$, and we recall that $G$ acts on $\Delta(G)$ by \begin{displaymath}
	g \mu \, \defeq \, \mu \circ {\lambda_{g^{-1}}} \qquad (g \in G, \, \mu \in \Delta(G)) .
\end{displaymath} Furthermore, if $T$ is a Hausdorff topological space, then we let $\Prob(T)$ denote the set of all inner regular Borel probability measures on $T$. Of course, if $T$ is a discrete topological space, then there exists a natural bijection between $\Prob(T)$ and the set $\{ \mu \in \ell^{1}(T) \mid \mu \geq 0, \, \Vert \mu \Vert_{1} = 1 \}$.

We now turn to functional-analytic aspects of uniform spaces, as developed by Pachl~\cite{PachlBook}. Let $X$ be a uniform space. Then $\UCB(X)$ denotes the commutative unital Banach algebra of all uniformly continuous bounded real-valued functions on $X$. The \emph{Samuel compactification} of $X$, denoted by $\Samuel(X)$, is defined as the compact space of all unital ring homomorphisms from $\UCB(X)$ to $\R$, equipped with the weak-$\ast$ topology, i.e., the Gel'fand spectrum of $\UCB(X)$. Moreover, we define $\Mean(X)$ as the compact space of all means (i.e., positive unital linear forms) on $\UCB(X)$, again equipped with the weak-$\ast$ topology. Note that $\Mean(X)$ is naturally homeomorphic to the space of inner regular Borel probability measures on $\Samuel(X)$. Following~\cite[Definition~1.19]{PachlBook}, a subset $B \subseteq \UCB(X)$ will be called \emph{UEB} (short for \emph{uniformly equicontinuous, bounded}) if \begin{itemize}
	\item[---\,] $B$ is bounded in the supremum norm, and
	\item[---\,] $B$ is uniformly equicontinuous, i.e., for each $\varepsilon \in \R_{>0}$ there is an entourage $U$ of $X$ such that \begin{displaymath}
		\qquad \forall f \in B \ \forall (x,y) \in U \colon \qquad \vert f(x) - f(y) \vert \, \leq \, \varepsilon .
	\end{displaymath}
\end{itemize} The set of all UEB subsets of $\UCB(X)$, denoted by $\UEB(X)$, constitutes a convex vector bornology on $\UCB(X)$. The induced locally convex topology on $\UCB(X)^{\ast}$, i.e., the one generated by the set of pseudo-norms of the form \begin{displaymath}
	p_{B} \colon \, \UCB(X)^* \, \longrightarrow \, \R_{\geq 0}, \quad \mu \, \longmapsto \, \sup\nolimits_{f \in B} \vert \langle \mu,f \rangle \vert
\end{displaymath}
where $B \in \UEB(X)$, is called the \emph{UEB uniformity}. We will not distinguish in notation between the set $\Delta(G)$ and its image under the map \begin{displaymath}
    \Delta (X) \, \longrightarrow \, \Mean(X), \quad \mu \, \longmapsto \, [f \mapsto \langle \mu,f\rangle]
\end{displaymath} (although this map is injective if and only if $X$ is Hausdorff). We define $\Meanu(X)$ to be the closure of $\Delta(X)$ in $\UCB(X)^{\ast}$ with respect to the UEB topology. A mean $\mu \in \Mean(X)$ belongs to $\Meanu(X)$ if and only if, for every $B \in \UEB(X)$, the restriction $\mu\vert_{B} \colon B \to \R$ is continuous with respect to the topology of pointwise convergence on $B$ (see~\cite[Theorem~6.6(2)]{PachlBook}). In~\cite{PachlBook}, the elements of $\Meanu(X)$ are called \emph{uniform measures}.

\begin{remark}\label{remark:weak.star.vs.ueb.topology} Let $X$ be any uniform space. On the set $\Meanu(X)$, the UEB topology coincides with the weak-$\ast$ topology inherited from $\UCB(X)^{\ast}$. This follows by~\cite[Corollary 6.13]{PachlBook}, or alternatively by~\cite[Corollary 6.15]{PachlBook}. \end{remark}

If $X$ is a Hausdorff uniform space, then \begin{displaymath}
    \Prob(X) \, \longrightarrow \, \Meanu(X), \quad \mu \, \longmapsto \, \left[f \mapsto \int f \, \mathrm{d} \mu\right]
\end{displaymath} is a well-defined injection~\cite[Lemma~5.2, Theorem~5.3]{PachlBook}, and we will view $\Prob(X)$ as a subset of $\Meanu(X)$ by means of this map.

\begin{remark}\label{remark:discrete.uniformity} If $X$ is a discrete uniform space, then $\Meanu(X) = \Prob(X)$ and the UEB uniformity on $\Prob(X)$ coincides with the one generated by $\Vert \cdot \Vert_{1}$. For instance, this follows by~\cite[Lemma~5.2, Theorem 5.28]{PachlBook}. \end{remark}

Finally, let $G$ be a topological group -- throughout this manuscript, all topological groups are assumed to be Hausdorff. Then $\Neigh(G)$ denotes the neighborhood filter at the neutral element $e \in G$. The uniformity \begin{displaymath}
	\left\{ E \subseteq G \times G \left\vert \, \exists U \in \Neigh(G) \, \forall x,y \in G \colon \, xy^{-1}\! \in U \Longrightarrow \, (x,y) \in E \right\} \right.\! 
\end{displaymath} is called the \emph{right uniformity} of $G$, and we let $G_{\Rsh}$ denote the uniform space consisting of the set $G$ equipped with that uniformity. We observe that a function $f \in \ell^{\infty}(G)$ belongs to $\RUCB(G) \defeq \UCB(G_{\Rsh})$ if and only if \begin{displaymath}
	\forall \varepsilon \in \R_{>0} \, \exists U \in \Neigh(G) \, \forall u \in U\colon \quad \Vert f-(f\circ \lambda_{u}) \Vert_{\infty} \leq \varepsilon .
\end{displaymath} In general, a norm-bounded set $B \subseteq \ell^{\infty}(G)$ belongs to $\RUEB (G) \defeq \UEB (G_{\Rsh})$ if and only if \begin{displaymath}
	\forall \varepsilon \in \R_{>0} \, \exists U \in \Neigh(G) \ \forall u \in U \ \forall f \in B \colon \quad \Vert f-(f\circ \lambda_{u}) \Vert_{\infty} \leq \varepsilon .
\end{displaymath} Analogously, \begin{displaymath}
	\left\{ E \subseteq G \times G \left\vert \, \exists U \in \Neigh(G) \, \forall x,y \in G \colon \, x^{-1}y\! \in U \Longrightarrow \, (x,y) \in E \right\} \right.\! 
\end{displaymath} is called the \emph{left uniformity} of $G$, and $G_{\Lsh}$ denotes the uniform space consisting of the set $G$ furnished with the left uniformity. Similarly as above, a function $f \in \ell^{\infty}(G)$ belongs to $\LUCB(G) \defeq \UCB(G_{\Lsh})$ if and only if \begin{displaymath}
	\forall \varepsilon \in \R_{>0} \, \exists U \in \Neigh(G) \, \forall u \in U\colon \quad \Vert f-(f\circ \rho_{u}) \Vert_{\infty} \leq \varepsilon ,
\end{displaymath} while a norm-bounded subset $B \subseteq \ell^{\infty}(G)$ belongs to $\LUEB (G) \defeq \UEB (G_{\Lsh})$ if and only if \begin{displaymath}
	\forall \varepsilon \in \R_{>0} \, \exists U \in \Neigh(G) \ \forall u \in U \ \forall f \in B \colon \quad \Vert f-(f\circ \rho_{u}) \Vert_{\infty} \leq \varepsilon .
\end{displaymath}

We conclude this section with a clarifying remark about UEB topologies for metrizable topological groups. By a \emph{compatible} metric on a metrizable topological space we mean one that generates the topology of the latter.

\begin{remark}\label{remark:mass.transportation} Let $G$ be a metrizable topological group. If $d$ is any right-invariant (resp., left-invariant) compatible bounded metric on $G$, then the UEB uniformity and the uniformity generated by the metric \begin{displaymath}
    d_{\mathrm{MT}} (\mu,\nu) \, \defeq \, \sup \{ \vert \mu(f) - \nu(f) \vert \mid f \in \Lip_{1}(G,d) \}
\end{displaymath} coincide on $\Mean(G_{\Rsh})$ (resp., $\Mean(G_{\Lsh})$). A proof can be found in~\cite[Lemma~2.6]{MR4152622}. \end{remark}

\section{Amenability of actions and exactness of topological groups} \label{sec:amenable}

A notion of amenability of actions of locally compact groups on compact topological spaces has been introduced by Claire Anantharaman-Delaroche; deviating from the existing literature we call this notion AD-amenability. The definition (in one of many equivalent forms) goes as follows:

\begin{definition}[Anantharaman-Delaroche]\label{definition:ad.amenability}
Let $G$ be a locally compact group equipped with a left Haar measure $\nu$ and let \begin{displaymath}
    \Prob(G,\nu) \, \defeq \, \left\{ \varphi \in L^{1}(G,\nu) \, \right\vert \varphi \geq 0, \, \Vert \varphi \Vert_{\nu,1} = 1 \} .
\end{displaymath} A continuous action of $G$ on a compact space $X$ is called \emph{AD-amenable} if there exists a net of $\sigma(C_0(G)^*,C_0(G))$-continuous maps $\mu_{\iota} \colon X \to \Prob(G,\nu)$ $(\iota \in I)$ such that \begin{displaymath}
    \sup\nolimits_{x \in X}\lVert g \mu_{\iota}(x) - \mu_{\iota}(gx)\rVert_{\nu,1}  \, \longrightarrow \, 0 \quad (\iota \to I)
\end{displaymath} uniformly on compact subsets of $G$.
\end{definition}

Deviating again from the existing literature, we will now define a new notion of amenability of actions for general topological groups. This definition has a skew version which turns out to be equivalent to AD-amenability for actions of locally compact groups (see Theorem \ref{theorem:ad.vs.skew}), and if one additionally assumes the locally compact group to be second-countable, then all three notions become equivalent (see Theorem \ref{theorem:amenable.implies.AD-amenable}).

\begin{definition}\label{def:amenableaction} A continuous action of a topological group $G$ on a compact space space $X$ is called \emph{amenable} if there is a net of continuous maps $\mu_{\iota} \colon X \to \Meanu(G_{\Rsh})$ $(\iota \in I)$ such that, for every $g \in G$ and every $B \in \RUEB(G)$, \begin{displaymath}
    \sup\nolimits_{x \in X} \sup\nolimits_{f \in B} \vert (g\mu_{\iota}(x))(f) - \mu_{\iota}(gx)(f) \vert \, \longrightarrow \, 0 \quad (\iota \to I) .
\end{displaymath} Such a net will then be called a \emph{witness of amenability} for the action. A topological group is called \emph{exact} (or \emph{amenable at infinity}) if it admits an amenable continuous action on some non-empty compact space. \end{definition}

\begin{remark} Equivalently, a topological group $G$ is exact if and only if it acts amenably on its Samuel compactification $\Samuel(G_{\Rsh})$, resp.~its universal minimal flow $\UMF(G)$.
\end{remark}

To avoid repetition in later arguments, let us record the following abstract observation. Recall that, for any topological space $X$ and any topological vector space $E$, the algebraic tensor product $C(X) \otimes E$ naturally embeds into $C(X,E)$, where \begin{displaymath}
    f \otimes z \colon \, X \, \longrightarrow \, E, \quad x \, \longmapsto \, f(x)z
\end{displaymath} for all $f \in C(X)$ and $z \in E$.

\begin{lemma}\label{lemma:abstract.approximation} Let $X$ be a compact space an let $E$ be a locally convex space. Let $D \subseteq E$ and $K \defeq \overline{D}$. Then \begin{displaymath}
    \left\{ \sum\nolimits_{i=1}^{n} f_{i} \otimes z_{i} \left\vert \, n \in \N_{>0}, z_{1},\ldots,z_{n} \in D, \, f_{1},\ldots,f_{n} \in C(X,[0,1]), \, \sum\nolimits_{i=1}^{n} f_{i} = 1 \right\} \right.
\end{displaymath} is dense in $C(X,K)$ with respect to the topology of uniform convergence. \end{lemma}

\begin{proof} Let $\psi \in C(X,K)$ and let $U$ be a convex zero neighborhood in $E$. Since $X$ is compact, there exist a finite open cover $\mathcal{V}$ of $X$ and a map $h \colon \mathcal{V} \to D$ such that $\psi (x) \in U + h(V)$ for all $x \in V$ and $V \in \mathcal{V}$. Choose any continuous partition of the unity subordinate to $\mathcal{V}$, i.e., a family of continuous functions $f_{V} \colon X \to [0,1]$ ($V \in \mathcal{V}$) with $\sum_{V \in \mathcal{V}} f_{V} = 1$ and $\overline{\spt (f_{V})} \subseteq V$ for each $V \in \mathcal{V}$. Consider $\varphi \defeq \sum\nolimits_{V \in \mathcal{V}} f_{V} \otimes h(V)$. Then, for each $x \in X$, \begin{align*}
	\psi (x) - \varphi (x) \, &= \, \sum\nolimits_{V \in \mathcal{V}} f_{V}(x)\psi (x) - \sum\nolimits_{V \in \mathcal{V}} f_{V}(x)h (V) \\
    &= \, \sum\nolimits_{V \in \mathcal{V}} f_{V}(x) (\psi (x) - h(V)) \, \in \, U 
\end{align*} due to $U$ being convex. This proves our claim. \end{proof}

\begin{lemma}\label{lemma:cut.off} A continuous action of a topological group $G$ on a compact space $X$ is amenable if and only if there exists a net of continuous maps $\mu_{\iota} \colon X \to \Delta(G)$ $(\iota \in I)$ such that \begin{itemize}
    \item[$(1)$] $\bigcup\nolimits_{x \in X} \spt (\mu_{\iota}(x))$ is finite for every $\iota \in I$, and
    \item[$(2)$] for every $g \in G$ and every $B \in \RUEB(G)$, \begin{displaymath}
    \qquad \sup\nolimits_{x \in X} \sup\nolimits_{f \in B} \vert (g\mu_{\iota}(x))(f) - \mu_{\iota}(gx)(f) \vert \, \longrightarrow \, 0 \quad (\iota \to I) .
\end{displaymath} \end{itemize} \end{lemma}

\begin{proof} Since $\overline{\Delta(G)} = \Meanu(G_{\Rsh})$ in the locally convex space $E \defeq \UCB(G_{\Rsh})^{\ast}$ endowed with the UEB topology, the claim follows by a straightforward application of Lemma~\ref{lemma:abstract.approximation}. \end{proof}

The following theorem provides the familiar characterization of amenability in the setting of amenable actions.

\begin{theorem}\label{thm:amenablegroup} Let $G$ be a topological group. The following are equivalent. \begin{itemize}
    \item[$(1)$] $G$ is amenable.
    \item[$(2)$] The action of $G$ on a singleton space is amenable.
    \item[$(3)$] Every continuous action of $G$ on a compact space is amenable.
\end{itemize} \end{theorem}

\begin{proof} This is a direct consequence of~\cite[Theorem~3.2]{MR3809992}. \end{proof}

We now turn to a skew version of the concepts introduced above.

\begin{definition}\label{def:skewamenableaction} A continuous action of a topological group $G$ on a compact space space $X$ is said to be \emph{skew-amenable} if there exists a net of continuous maps $\mu_{\iota} \colon X \to \Meanu(G_{\Lsh})$ $(\iota \in I)$ such that, for every $g \in G$ and for every $B \in \LUEB(G)$, \begin{displaymath}
    \sup\nolimits_{x \in X} \sup\nolimits_{f \in B} \vert (g\mu_{\iota}(x))(f) - \mu_{\iota}(gx)(f) \vert \, \longrightarrow \, 0 \quad (\iota \to I) .
\end{displaymath} Such a net will be called a \emph{witness of skew-amenability} for the action. A topological group is called \emph{skew-exact} (or \emph{skew-amenable at infinity}) if it admits a skew-amenable continuous action on some non-empty compact space. \end{definition}

\begin{remark} Equivalently, a topological group $G$ is skew-exact if and only if it acts skew-amenably on its own Samuel compactification resp.\ its universal minimal flow $\UMF(G)$.
\end{remark}

\begin{lemma}\label{lemma:skew.cut.off} A continuous action of a topological group $G$ on a compact space space $X$ is skew-amenable if and only if there exists a net of continuous maps $\mu_{\iota} \colon X \to \Delta(G)$ $(\iota \in I)$ such that \begin{itemize}
    \item[$(1)$] $\bigcup\nolimits_{x \in X} \spt (\mu_{\iota}(x))$ is finite for every $\iota \in I$, and
    \item[$(2)$] for every $g \in G$ and every $B \in \LUEB(G)$, \begin{displaymath}
    \qquad \sup\nolimits_{x \in X} \sup\nolimits_{f \in B} \vert (g\mu_{\iota}(x))(f) - \mu_{\iota}(gx)(f) \vert \, \longrightarrow \, 0 \quad (\iota \to I) .
\end{displaymath} \end{itemize} \end{lemma}

\begin{proof} Since $\overline{\Delta(G)} = \Meanu(G_{\Lsh})$ in the locally convex space $E \defeq \UCB(G_{\Lsh})^{\ast}$ endowed with the UEB topology, the desired conclusion follows by an application of Lemma~\ref{lemma:abstract.approximation}. \end{proof}

We now record a useful description of (skew-)amenability of actions which follows from the above discussion.

\begin{lemma}\label{lemma:witnesses.in.function.spaces}
A continuous action of a topological group $G$ on a compact space $X$ is amenable resp. skew-amenable if and only if there exists a net $\mu_\iota\in C(X,\Delta(G))$ such that $\bigcup\nolimits_{x \in X} \spt (\mu_{\iota}(x))$ is finite for every $\iota \in I$ and for every $g\in G$ we have
\[
(g\otimes g^{-1})(\mu_\iota)-\mu_\iota \to 0,\quad \iota\to I
\]
in $C(X,\UCB(G_{\Rsh})^{\ast})$ resp. $C(X,\UCB(G_{\Lsh})^{\ast})$ with the topology of uniform convergence, the spaces $\UCB(G_{\Rsh})^{\ast}$ resp. $\UCB(G_{\Lsh})^{\ast}$ being endowed with the UEB topologies.
\end{lemma}

In analogy to Theorem \ref{thm:amenablegroup}, we obtain a characterization of skew-amenability of a topological group.

\begin{theorem} Let $G$ be a topological group. The following are equivalent. \begin{itemize}
    \item[$(1)$] $G$ is skew-amenable.
    \item[$(2)$] The action of $G$ on a singleton space is skew-amenable.
    \item[$(3)$] Every continuous action of $G$ on a compact space is skew-amenable.
\end{itemize} \end{theorem}

\begin{proof} This is a direct consequence of~\cite[Theorem~5.1]{JuschenkoSchneider}, which is a special case of~\cite[Theorem~3.2]{pachl}. \end{proof}

In order to compare the notion of AD-amenability to (skew-)amenability in the sense of Definition~\ref{def:amenableaction} (resp., \ref{def:skewamenableaction}), we turn to the natural maps constructed in the following remark.

\begin{remark}\label{remark:haar.measure} Let $G$ be a locally compact group and let $\nu$ be a left Haar measure on $G$. Then the map $\Xi_{\nu}^{\Rsh} \colon \Prob(G,\nu) \to \Meanu(G_{\Rsh})$ given by \begin{displaymath}
    \Xi_{\nu}^{\Rsh}(\varphi)(f) \, \defeq \, \int f(x)\varphi(x) \, \mathrm{d}\nu(x) \qquad (\varphi \in \Prob(G,\nu), \, f \in \RUCB(G))
\end{displaymath} is well defined, $G$-left-equivariant, and uniformly continuous with respect to the UEB uniformity on $\Meanu(G_{\Rsh})$ and the uniformity induced by $\Vert \cdot \Vert_{\nu,1}$ on $\Prob(G,\nu)$. Likewise, the map $\Xi_{\nu}^{\Lsh} \colon \Prob(G,\nu) \to \Meanu(G_{\Lsh})$ given by \begin{displaymath}
    \Xi_{\nu}^{\Lsh}(\varphi)(f) \, \defeq \, \int f(x)\varphi(x) \, \mathrm{d}\nu(x) \qquad (\varphi \in \Prob(G,\nu), \, f \in \LUCB(G))
\end{displaymath} is well defined, $G$-left-equivariant, and uniformly continuous with respect to the UEB uniformity on $\Meanu(G_{\Lsh})$ and the uniformity induced by $\Vert \cdot \Vert_{\nu,1}$ on $\Prob(G,\nu)$. \end{remark}

We do not know if, for a locally compact group $G$ equipped with a left Haar measure $\nu$, the $\sigma(C_0(G)^*,C_0(G))$-topology on $\Prob(G,\nu)$ agrees with the $\UEB$ topology inherited from $\Mean_u(G_{\Rsh})$ (or $\Mean_u(G_{\Lsh})$, resp.). In turn, it is unclear whether $\Xi_{\nu}^{\Rsh}$ (resp., $\Xi_{\nu}^{\Lsh}$) is $\sigma(C_0(G)^*,C_0(G))$-to-$\UEB$ continuous. However, we will see that for the definition of the AD-amenability of an action, the difference of these topologies is inessential. The following proposition is a corollary to \cite[Proposition 2.2]{AD23amenabilityexactness} and will be used later in both Theorem~\ref{theorem:ad.vs.skew} and Corollary~\ref{corollary:skew.implies.straight}. 

\begin{proposition}\label{proposition:ad.definition} Let $G$ be a locally compact group acting continuously on a compact space $X$. If $G \curvearrowright X$ is AD-amenable, then there exists a witness $(\mu_{\iota})_{\iota \in I}$ of AD-amenability for $G \curvearrowright X$ such that, for each $\iota \in I$, the map ${\Xi_{\nu}^{\Lsh}} \circ {\mu_{\iota}} \colon X \to \Mean_u(G_{\Rsh})$ (resp., ${\Xi_{\nu}^{\Lsh}} \circ {\mu_{\iota}} \colon X \to \Mean_u(G_{\Lsh})$) is $\UEB$-continuous. \end{proposition}

The proof of Proposition~\ref{proposition:ad.definition} will be given in the appendix.

If the group $G$ is second-countable, we can produce an amenable action on a second-countable compact space out of an arbitrary amenable action.

\begin{proposition}\label{prop:separability-reduction}
Let $G\curvearrowright X$ be an amenable action of a second-countable topological group $G$ on a compact space $X$. Then there exists a second-countable compact space $Y$ with an amenable action $G\curvearrowright Y$ together with a $G$-equivariant continuous surjection $q\colon X\to Y$.
\end{proposition}

\begin{proof} Since $G$ is second-countable, $G$ is separable and metrizable. While the former asserts the existence of a countable dense subset $S \subseteq G$, the latter implies metrizability of the UEB topology on $\Meanu(G_{\Rsh})$ by Remark~\ref{remark:mass.transportation}, which entails metrizability of $C(X,\Meanu(G_{\Rsh}))$ with respect to the topology of uniform convergence. Since $G \curvearrowright X$ is amenable, using Lemma~\ref{lemma:cut.off} and Lemma~\ref{lemma:witnesses.in.function.spaces} we thus find a sequence $(\mu_{n})_{n \in \N}$ in $C(X,\Delta(G))$ such that each $\mu_{n}$ is of the form
\begin{equation}\label{eq:finite_sums_mu_i}
    \mu_{n} \, = \, \sum\nolimits_{i=1}^{m_{n}} f_{n,i}\otimes z_{n,i}, \quad f_{n,i} \in C(X), \quad z_{n,i} \in \Delta(G) ,
\end{equation}
 and
\begin{equation}\label{eq:almost.equivariance}
    \forall g \in S \colon \qquad \left(g\otimes g^{-1}\right)\!\mu_{n} - \mu_{n} \, \longrightarrow \, 0 \quad (n \to \infty) .
\end{equation} 
It is not difficult to see that the homomorphism from $G$ to the automorphism group of the uniform space $C(X,\Meanu(G_{\Rsh}))$ induced by the action is continuous with respect to the topology of uniform convergence on the latter group. As $S$ is dense in $G$, we thus infer from~\eqref{eq:almost.equivariance} that \begin{equation}\label{eq:amenability}
    \forall g \in G \colon \qquad \left(g\otimes g^{-1}\right)\!\mu_{n} - \mu_{n} \, \longrightarrow \, 0 \quad (n \to \infty) .
\end{equation} Let us denote by $A$ the $G$-invariant unital $C^{\ast}$-subalgebra of $C(X)$ generated by the countable subset $\{ f_{n,i} \mid n \in \N, \, i \in \{ 1,\ldots,m_{n} \} \}$. As $G$ is separable, so is $A$. Applying Gel'fand duality to the inclusion $A\subseteq C(X)$, we obtain a second-countable compact space $Y$ together with a continuous action $G\lacts Y$ and a $G$-equivariant surjection $X\to Y$. From~\eqref{eq:finite_sums_mu_i} and~\eqref{eq:amenability} we deduce that $G\lacts Y$ is amenable. \end{proof}

\section{Comparison in the locally compact case}
\label{sec:comparison}

We prove that for an action $G\lacts X$ of a locally compact group $G$ on a compact space $X$, AD-amenability  is equivalent to skew-amenability (Theorem~\ref{theorem:ad.vs.skew}), and if $G$ is assumed second-countable, it is also equivalent to amenability \ref{theorem:amenable.implies.AD-amenable}. 

\subsection{AD-amenability vs.\ skew-amenability}

We start off with some relevant prerequisites. The following lemma is well known.

\begin{lemma}[cf.~\cite{HewittRoss}, (20.4) Theorem, p.~285]\label{lemma:representation} Let $G$ be a locally compact group with a left Haar measure $\nu$. Then the representation $\pi \colon G \to \Iso(L^{1}(G,\nu))$ defined by \begin{displaymath}
    \pi(g)(f) \, \defeq \, f \circ \lambda_{g^{-1}} \qquad \left(g \in G, \, f \in L^{1}(G,\nu)\right)
\end{displaymath} is strongly continuous. \end{lemma}

Recall that, if $G$ is any group, then the real group algebra of $G$ is the $\R$-vector space $\R[G]$ equipped with the bilinear multiplication given by \begin{displaymath}
    \alpha \beta \, \defeq \, \sum\nolimits_{g \in \spt \alpha} \alpha(g)\! \left(\beta \circ {\lambda_{g^{-1}}}\right) \qquad (\alpha,\beta \in \R[G]) . 
\end{displaymath}

Now, consider a locally compact group $G$ and let $\nu$ be a left Haar measure on $G$. Then the representation from Lemma~\ref{lemma:representation} gives rise to an $\R[G]$-module structure on $L^{1}(G,\nu)$, namely defined via \begin{displaymath}
    \alpha \varphi \, \defeq \, \sum\nolimits_{g \in \spt \alpha} \alpha(g) \!\left(\varphi \circ {\lambda_{g^{-1}}}\right) \qquad \left(\alpha \in \R[G], \, \varphi \in L^{1}(G,\nu)\right) .
\end{displaymath} Moreover, for all $\varphi \in L^{1}(G,\nu)$ and $f \in L^{\infty}(G,\nu)$, let us define \begin{displaymath}
    \langle \varphi,f \rangle \, \defeq \, \int f(x)\varphi(x) \, \mathrm{d}\nu(x)
\end{displaymath} and \begin{displaymath}
    {}_{\varphi}f \colon \, G \, \longrightarrow \, \R , \quad g \, \longmapsto \, \langle \varphi,f \circ {\lambda_{g}} \rangle = \int f(gx)\varphi(x) \, \mathrm{d}\nu(x) .
\end{displaymath} We record two elementary, but useful properties of these objects. The first lemma uses some notation introduced in Section~\ref{section:preliminaries}.

\begin{lemma}\label{lemma:shift} Let $G$ be a locally compact group equipped with a left Haar measure $\nu$. For all $\alpha \in \R[G]$, $\varphi \in L^{1}(G,\nu)$ and $f \in L^{\infty}(G,\nu)$, \begin{displaymath}
    \langle \alpha \varphi, f \rangle \, = \, \langle \alpha ,{}_{\varphi}f \rangle .
\end{displaymath} \end{lemma}

\begin{proof} If $\alpha \in \R[G]$, $\varphi \in L^{1}(G,\nu)$ and $f \in L^{\infty}(G,\nu)$, then \begin{align*}
    &\langle \alpha \varphi, f \rangle = \int f(x)(\alpha \varphi)(x) \, \mathrm{d}\nu(x) = \sum\nolimits_{g \in \spt \alpha} \alpha(g) \int f(x)\varphi\!\left(g^{-1}x\right) \! \, \mathrm{d}\nu(x) \\
    &\ \, = \sum\nolimits_{g \in \spt \alpha} \alpha(g)\int f(gx)\varphi(x) \, \mathrm{d}\nu(x) = \sum\nolimits_{g \in \spt \alpha} \alpha(g)({}_{\varphi}f)(g) = \langle \alpha,{}_{\varphi}f \rangle . \qedhere
\end{align*} \end{proof}

\begin{lemma}\label{lemma:lueb} Let $G$ be a locally compact group equipped with a left Haar measure $\nu$. Let $\varphi \in L^{1}(G,\nu)$. Then \begin{displaymath}
    \{ {}_{\varphi}f \mid f \in L^{\infty}(G,\nu), \, \Vert f \Vert_{\nu,\infty} \leq 1 \} \, \in \, \LUEB(G) .
\end{displaymath} \end{lemma}

\begin{proof} Let us abbreviate $B \defeq \{ {}_{\varphi}f \mid f \in L^{\infty}(G,\nu), \, \Vert f \Vert_{\nu,\infty} \leq 1 \}$. First of all, if $f \in L^{\infty}(G,\nu)$ and $\Vert f \Vert_{\nu,\infty} \leq 1$, then \begin{displaymath}
    \Vert {}_{\varphi}f \Vert_{\infty} \, = \, \sup\nolimits_{g \in G} \left\lvert \int f(gx) \varphi(x) \, \mathrm{d}\nu(x) \right\rvert \, \leq \, \Vert f \Vert_{\nu,\infty} \cdot \Vert \varphi \Vert_{\nu,1} \, \leq \, \Vert \varphi \Vert_{\nu,1} .
\end{displaymath} Hence, $B$ is a bounded subset of $(\ell^{\infty}(G),\Vert \cdot \Vert_{\infty})$. To prove left uniform equicontinuity, let $\varepsilon \in \R_{>0}$. By Lemma~\ref{lemma:representation}, we find $U \in \Neigh(G)$ such that \begin{displaymath}
    \forall u \in U \colon \qquad \left\lVert \left(\varphi \circ \lambda_{u^{-1}}\right) - \varphi \right\rVert_{\nu,1} \, \leq \, \varepsilon .
\end{displaymath} Now, if $u \in U$ and $f \in L^{\infty}(G,\nu)$ with $\Vert f \Vert_{\nu,\infty} \leq 1$, then \begin{align*}
    \left\lvert ({}_{\varphi}f)(gu) - ({}_{\varphi}f)(g) \right\rvert \, & = \, \left\lvert \int f(gux)\varphi(x) \, \mathrm{d}\nu(x) - \int f(gx)\varphi(x) \, \mathrm{d}\nu(x) \right\rvert \\
    &= \, \left\lvert \int f(gx)\varphi\!\left(u^{-1}x\right)\! \, \mathrm{d}\nu(x) - \int f(gx)\varphi(x) \, \mathrm{d}\nu(x) \right\rvert \\
    &\leq \, \int \lvert f(gx)\rvert\! \left\lvert\varphi\!\left(u^{-1}x\right)\! - \varphi(x)\right\rvert \! \, \mathrm{d}\nu(x) \\
    &\leq \, \left\lVert f \circ {\lambda_{g}} \right\rVert_{\nu,\infty} \cdot \left\lVert \left( \varphi \circ {\lambda_{u^{-1}}} \right) - \varphi \right\rVert_{\nu,1} \, \leq  \, \varepsilon
\end{align*} for all $g \in G$, i.e., \begin{displaymath}
    \left\lVert ( ({}_{\varphi}f) \circ \rho_{u})- ({}_{\varphi}f) \right\rVert_{\infty} \, \leq \, \varepsilon .
\end{displaymath} This shows that $B \in \LUEB(G)$. \end{proof}

Given a locally compact group $G$ and a left Haar measure $\nu$ on $G$, we let \begin{displaymath}
    \left. \Prob(G,\nu) \, \defeq \, \left\{ \varphi \in L^{1}(G,\nu) \, \right\vert \varphi \geq 0, \, \Vert \varphi \Vert_{\nu,1} = 1 \right\} .
\end{displaymath}

\begin{proposition}\label{proposition:ad.vs.skew} Let $G$ be a locally compact group equipped with a left Haar measure $\nu$. Let $\varphi \in \Prob(G,\nu)$ and consider \begin{displaymath}
    B \, \defeq \, \{ {}_{\varphi}f \mid f \in L^{\infty}(G,\nu), \, \Vert f \Vert_{\nu,\infty} \leq 1 \} \, \stackrel{\ref{lemma:lueb}}{\in} \, \LUEB(G) .
\end{displaymath} Then \begin{displaymath}
    \Delta(G) \, \longrightarrow \, \Prob(G,\nu), \quad \mu \, \longmapsto \, \mu\varphi
\end{displaymath} is $G$-left-equivariant and \begin{displaymath}
    \forall \mu,\mu' \in \Delta(G) \colon \quad \Vert \mu\varphi - \mu'\varphi \Vert_{\nu,1} \, = \, p_{B}(\mu-\mu') .
\end{displaymath} \end{proposition}

\begin{proof} First of all, it is easy to see that the given map is indeed well defined. Moreover, since $L^{1}(G,\nu)$ is an $\R[G]$-module, \begin{displaymath}
    (g\mu)\varphi \, = \, (\delta_{g}\mu)\varphi \, = \, \delta_{g}(\mu\varphi) \, = \, (\mu\varphi) \circ \lambda_{g^{-1}}
\end{displaymath} for all $g \in G$ and $\mu \in \Delta(G)$, i.e., the mapping is $G$-left-equivariant. Finally, if $\mu,\mu' \in \Delta(G)$, then \begin{align*}
    \Vert \mu\varphi - \mu'\varphi \Vert_{\nu,1} \, &= \, \sup \{ \langle \mu\varphi - \mu'\varphi, f \rangle \mid f \in L^{\infty}(G,\nu), \, \Vert f \Vert_{\nu,\infty} \leq 1 \} \\
    &\stackrel{\ref{lemma:shift}}{=} \, \sup \{ \langle \mu - \mu', {}_{\varphi}f \rangle \mid f \in L^{\infty}(G,\nu), \, \Vert f \Vert_{\nu,\infty} \leq 1 \} \\
    &= \, p_{B}(\mu-\mu') .\qedhere
\end{align*} \end{proof}

\begin{theorem}\label{theorem:ad.vs.skew} A continuous action of a locally compact group on a compact space is AD-amenable if and only if it is skew-amenable. \end{theorem}

\begin{proof} Let $G$ be a locally compact group and let $\nu$ be a left Haar measure on $G$. Consider a continuous action of $G$ on a compact space $X$.

($\Longrightarrow$) Suppose that $G \curvearrowright X$ is AD-amenable. By Proposition~\ref{proposition:ad.definition}, there exists a witness $(\mu_{\iota})_{\iota \in I}$ of AD-amenability for $G \curvearrowright X$ such that, for each $\iota \in I$, the map ${\Xi_{\nu}^{\Lsh}} \circ {\mu_{\iota}} \colon X \to \Mean_u(G_{\Lsh})$ is $\UEB$-continuous. It now follows by Remark~\ref{remark:haar.measure} that the net $({\Xi_{\nu}^{\Lsh}} \circ {\mu_{\iota}})_{\iota \in I}$ is a witness of skew-amenability for $G \curvearrowright X$.


($\Longleftarrow$) Fix any $\varphi \in \Prob(G,\nu)$. According to Proposition~\ref{proposition:ad.vs.skew},  \begin{displaymath}
    \Delta(G) \, \longrightarrow \, \Prob(G,\nu), \quad \mu \, \longmapsto \, \mu\varphi
\end{displaymath} is well defined, $G$-left-equivariant, and uniformly continuous (with respect to the UEB uniformity on $\Delta(G) \subseteq \Meanu(G_{\Lsh})$ and the uniformity induced by $\Vert \cdot \Vert_{\nu,1}$ on $\Prob(G,\nu)$). Thus, if $\mu_{\iota} \colon X \to \Delta(G)$ $(\iota \in I)$ is a witness of skew-amenability for the action $G \curvearrowright X$ (in the sense of Lemma~\ref{lemma:skew.cut.off}), then the net \begin{displaymath}
    X \, \longrightarrow \, \Prob(G,\nu), \quad x \, \longmapsto \, \mu_{\iota}(x)\varphi \qquad (\iota \in I)
\end{displaymath} is a witness of AD-amenability for $G \curvearrowright X$. \end{proof}

\begin{corollary}\label{corollary:skew.implies.straight}
Every skew-amenable (or, equivalently, AD-amenable) continuous action of a locally compact group on a compact space is amenable.
\end{corollary}

\begin{proof} Thanks to Theorem~\ref{theorem:ad.vs.skew}, skew-amenability implies AD-amenability, which in turn implies amenability by an application of Proposition~\ref{proposition:ad.definition} and Remark~\ref{remark:haar.measure} (analogous to the application of Proposition~\ref{proposition:ad.definition} and Remark~\ref{remark:haar.measure} in the proof of Theorem~\ref{theorem:ad.vs.skew}($\Longrightarrow$)). \end{proof}

\subsection{AD-amenability vs.\ amenability for second-countable locally compact groups}

Our next goal is to prove that amenability implies AD-amenability for actions of locally compact second-countable groups. 

Let $G\curvearrowright X$ be a continuous action of a locally compact group on a compact space $X$. 
Many of the notions below can be reformulated using the language of locally compact groupoids, but for our purposes it will suffice to restrict ourselves to group actions. We will, however, use the suggestive notation $G\ltimes X$ to denote the product $G\times X$ equipped with the groupoid product $(h,gx)\circ (g,x) = (h,x)$. Notice that this groupoid structure corresponds to the direct product action of $G$ on $G\times X$ which we will use throughout.

\begin{definition}
A function $f\in C_b(G\ltimes X)$ is called right uniformly continuous if for every $\eps > 0$ there exists a neighborhood $U$ of the identity in $G$ such that for all $x\in X$, $g\in G$ and $u\in U$ we have $|f(g,x) - f(ug,x)| < \eps$. The $C^\ast$-algebra of right uniformly continuous functions on $G\ltimes X$ is denoted by $\RUCB(G\ltimes X)$.
\end{definition}

Given $f\in C_b(G\ltimes X)$, we denote by $f_x\in C_b(G)$ the function $g\mapsto f(g,x)$. It can be thought of as the restriction of $f$ to the source fiber of the groupoid. We observe that $f\in C_b(G\ltimes X)$ is in $\RUCB(G\ltimes X)$ if and only if for every $x\in X$ the function $f_x$ belongs to $\RUCB(G)$. Since $X$ is compact, the set $\{f_x\mid x\in X\}$ then belongs to $\RUEB(G)$.

\begin{definition}
The action $G\curvearrowright X$ is called \emph{measurewise amenable} if for every quasi-invariant measure $\mu$ on $X$ the $G$-von Neumann algebra $L^\infty(X)$ is amenable, i.e., there is a $G$-equivariant projection $L^\infty(G) \mathbin{\overline\otimes} L^\infty(X)\to L^\infty(X)$.
\end{definition}

The algebra of continuous functions with compact support $C_c(G\ltimes X)$ is naturally equipped with the convolution product
\[
(f_1\ast f_2)(g,x) \, = \, \int_{G} f_1(h,h^{-1}gx)f_2(h^{-1}g,x) d\nu(h),
\]
where $\nu$ is a left Haar measure on $G$ which we fix for the rest of the section. Given a quasi-invariant measure $\mu$ on $X$, we can complete it to the convolution algebra $L^\infty(X,\mu,L^1(G,\nu))$. We will denote by $D(G\ltimes X)$ the convex subset of $L^\infty(X,\mu,L^1(G,\nu))$ consisting of continuous non-negative functions $\phi$ with compact support such that $\norm{\phi(x)}_1 = 1$ for all $x\in X$. An approximate identity $e_n\in C_c(G,\nu)_+\subset L^1(G,\nu)_+$ with $\norm{e_n}_1 = 1$ for $L^1(G,\nu)$ naturally embeds into $D(G\ltimes X)$ (as functions constant in $x$) and gives an approximate identity for the convolution algebra $L^\infty(X,\mu,L^1(G,\nu))$.

\begin{lemma}\label{lem:amenable-impl-expectation}
Let $G\curvearrowright X$ be amenable and let $\mu$ be a quasi-invariant measure on $X$. Then there exists a unital positive $G$-equivariant contraction \begin{displaymath}
    E\colon \, \RUCB(G\ltimes X) \, \longrightarrow \, L^\infty(X,\mu) .
\end{displaymath} \end{lemma}

\begin{proof}
Let $\mu_\iota$ be a witness of amenability of the action. We now define unital positive contractions
\[
E_\iota\colon \RUCB(G \ltimes X)\to L^\infty(X,\mu),\quad (E_\iota(f))(x)\coloneqq (\mu_\iota(x))(f_x).
\]
Taking a point-ultraweak limit point of $E_\iota$, we get a unital positive contraction $E\colon \RUCB(G\ltimes X)\to L^\infty(X,\mu)$. It is routine to check that it is $G$-equivariant: by definition we have \begin{align*}
    (E_\iota(g\cdot f))(x) \, &= \, (\mu_\iota(x))(g\cdot f_{g^{-1}x}), \\
    (g\cdot E_\iota(f))(x) \, &= \, E_\iota(f)(g^{-1}x) \, = \, (\mu_\iota(g^{-1}x))(f_{g^{-1}x}) .
\end{align*}
Since $\mu_\iota$ is a witness of amenability and the family $(f_x)_{x\in X}$ is in $\RUEB(G)$, we obtain that $E$ is $G$-equivariant, as desired.
\end{proof}

\begin{lemma}\label{lem:convolution-becomes-RUCB}
For all $f\in L^\infty(G\ltimes X,\nu\times \mu) $ and $\phi\in D(G\ltimes X)$, we have $\phi\ast f\in \RUCB(G\ltimes X)$.
\end{lemma}

\begin{proof}
By definition, \begin{displaymath}
    (\phi\ast f)(g,x) \, = \, \int_{G} \phi(h,h^{-1}gx)f(h^{-1}g,x)d\nu(h).
\end{displaymath} Since $\norm{\phi\ast f}_\infty \leq \norm{\phi}_1\norm{f}_\infty$, and elements of $D(G\ltimes X)$ represented by continuous functions with compact support are dense in the latter, we can assume that $\phi$ is continuous and has compact support. Then for every $u\in G$ we have \begin{align*}
    (\phi\ast f)(ug,x) \, &= \, \int_{G} \phi(h,h^{-1}ugx)f(h^{-1}ug,x)d\nu(h) \\
    &= \, \int_{G} \phi(u^{-1}h,h^{-1}gx)f(h^{-1}g,x)d\nu(h).
\end{align*} Since $\phi$ is continuous and has compact support, for every $\eps > 0$ there is a neighbourhood $U$ of the identity in $G$ such that for all $g\in G$, all $u\in U$ and all $x\in X$ we have $|\phi(u^{-1}g,x) - \phi(g,x)| < \eps$. The claim follows. 
\end{proof}

To prove that amenability implies AD-amenability for locally compact second-countable groups, we will rely on the following theorem.
\begin{theorem}[Buss--Echterhoff--Willett, {\cite[Theorem 3.27]{BussEchterhoffWillett}}]\label{thm:buss-echterhoff-willett}
    The action $G\curvearrowright X$ of a second-countable locally compact group $G$ on a compact space $X$
    is AD-amenable if and only if it is measurewise amenable.
\end{theorem}

\begin{theorem}\label{theorem:amenable.implies.AD-amenable}
 Let $G$ be a second-countable locally compact group. The action $G\curvearrowright X$ is amenable if and only if it is AD-amenable.
\end{theorem}
\begin{proof}
Since AD-amenability implies amenability by Corollary \ref{corollary:skew.implies.straight}, we only need to prove the converse. To this end, we  reduce the situation to the case where the base space is second-countable. Indeed, by Proposition \ref{prop:separability-reduction}, there is a second-countable $G$-quotient $Y$ of $X$ on which the action is still amenable. We will now prove that $Y$ is AD-amenable.

By Theorem \ref{thm:buss-echterhoff-willett}, it suffices to show that the action on $Y$ is measurewise amenable. Let $\mu$ be a quasi-invariant measure on $Y$. By Lemma \ref{lem:amenable-impl-expectation}, we get a $G$-equivariant contraction $E\colon \RUCB(G\ltimes Y)\to L^\infty(Y,\mu)$, and we have to extend it to a $G$-equivariant contraction $L^\infty(G)\overline\otimes L^\infty(Y)\to L^\infty(Y)$.

Let $e_n\in D(G\ltimes Y)$ be an approximate identity for the convolution algebra $L^\infty(X,\mu,L^1(G,\nu))$, and take a free ultrafilter $\mc U$ on $\N$. For every $f\in L^\infty(G\ltimes Y)$, we define \begin{displaymath}
    E'(f) \, \defeq \, \lim\nolimits_{n\to\mc U} E(e_n\ast f),
\end{displaymath} the ultralimit being taken in the ultraweak topology. This is well defined, since $e_n\ast f \in \RUCB(G\ltimes Y)$ by Lemma \ref{lem:convolution-becomes-RUCB} and $E$ has norm $1$. By construction, $E'$ is a positive contraction, which has norm $1$ since $E'(1) = E(1) = 1$, as $e_n\ast 1 = 1$.

It remains to check equivariance. We first prove that $E'$ is equivariant with respect to convolution with an arbitrary $\psi\in L^\infty(Y,L^1
(G))$. Since $e_n$ is an approximate identity, we have
\[
\norm{e_n\ast\psi - \psi\ast e_n}_1 \, \longrightarrow \, 0 \quad (n \to \mc{U}),
\]
and therefore for every $f\in L^\infty(G\ltimes Y)$
\[
\norm{E(e_n\ast\psi\ast f) - E(\psi\ast e_n\ast f)} \, \longrightarrow \, 0 \quad (n \to \mc{U}).
\]
We thus get \begin{align*}
E'(\psi\ast f) \, &= \, \lim\nolimits_{n\to \mc U} E(e_n\ast\psi\ast f) \, = \, \lim\nolimits_{n\to \mc U} E(\psi\ast e_n\ast f) \\
&= \, \lim\nolimits_{n\to \mc U} \psi\ast E(e_n\ast f) \, = \, \psi\ast E'(f),
\end{align*} as desired.

Finally, we run another approximation argument to check that $E'$ is indeed $G$-equivariant. Fixing $g\in G$ and letting $e_n\in C_c(G,\nu)_+\subset L^1(G,\nu)$ be an approximate identity with $\norm{e_n}_1 = 1$, we let $e_n^g(h) \coloneqq \Delta(g^{-1})e_n(hg)$ and again embed it into $L^\infty(Y,L^1(G,\nu))$. We notice that for all $f\in L^\infty(Y,L^1(G,\nu))$ we have $e_n^g\ast f = e_n \ast \lambda_g(f)$.

As $e_n$ is an approximate identity for $L^1(G,\nu)$, it is routine to check that $e_n^g\ast \varphi \to \lambda_g(\varphi)$ in $L^\infty(Y,L^1(G,\nu))$ for all $\varphi\in L^\infty(Y,L^1(G,\nu))$. By duality, this implies  $e_n^g\ast f\to \lambda_g(f)$ ultraweakly for all $f\in L^\infty(G\ltimes Y,\mu\circ\nu)$. We now compute \begin{align*}
\lambda_g(E'(f)) \, &= \, \lim\nolimits_{n\to \mc U} e_n^g\ast E'(f) \, = \, \lim\nolimits_{n\to \mc U} E'(e_n^g\ast f) \\
&= \, \lim\nolimits_{n\to \mc U} E(e_n^g\ast f) \, = \, \lim\nolimits_{n\to \mc U} E(e_n\ast \lambda_g(f)) \, = \, E'(\lambda_g(f)),
\end{align*} as desired. We thus have checked that $G\lacts Y$ is measurewise amenable, and by Theorem \ref{thm:buss-echterhoff-willett} is follows to be AD-amenable. Since $Y$ is a $G$-quotient of $X$, we deduce that $G\lacts X$ is AD-amenable. This finishes the proof.
\end{proof}

The removal of separability assumptions can be sometimes technical or even impossible. Anyway, we believe that the answer to the following question should be positive.

\begin{question}
Is AD-amenability equivalent to amenability for actions of a locally compact group $G$ without the assumption that $G$ is second-countable? 
\end{question}


\section{Amenably embedded and well-placed subgroups}
\label{sec:emb}
We start with two definitions.
\begin{definition}[Monod] \label{def:amenemb} A topological subgroup $H$ of a topological group $G$ is said to be \emph{amenably embedded} in $G$ if the action of $H$ by left translations on $\Samuel (G_{\Rsh})$ is amenable. \end{definition}

Clearly, a topological group is exact if it is amenably embedded in itself.
\begin{remark}
Note that in contrast to the previous notions there does not seem to be a skew version of the previous definition, since the left-translation action of $H$ on $\Samuel(G_{\Lsh})$ is usually not even continuous.
\end{remark}

\begin{theorem}
    Let $H$ be an amenably embedded subgroup of a topological group $G$. If a continuous action $G \curvearrowright X$ on a compact space $X$ is amenable, then its restriction $H \curvearrowright X$ is amenable as well. In particular, if $G$ is amenable, then $H$ is amenable.
    \end{theorem}
    \begin{proof}
    Let $\mu_i\colon X\to \Delta(G)$, $i\in I$, be a witness of amenability of the action $G \curvearrowright X$. Since $H$ is amenably embedded, we also get a witness of amenability for the action $H\curvearrowright \Samuel (G_{\Rsh})$ in form of a net $\nu_j\colon \Samuel (G_{\Rsh})\to \Delta(H)$, $j\in J$. By the universal property of the Samuel compactification, each restriction $\nu_j\colon G \to \Delta(H)$ is right-uniformly continuous. Slightly abusing notation, we keep this notation for the natural linear extension $\nu_j\colon \mathbb R[G]\to \mathbb R[H]$, and notice that because of right uniform continuity of $\nu_j$ this map is continuous with respect to the UEB topologies inherited from $\RUCB(G)^\ast$ resp. $\RUCB(H)^\ast$.
    
    We will now construct a map $j\colon I\to J$ such that the net $\nu_{j(i)}\circ \mu_i$ forms a witness of amenability for the action $H\curvearrowright X$.
    
    Since $\mu_i$ is a witness of amenability, for each $h\in H$ we have
    \[
        \left(h\otimes h^{-1}\right)\!\mu_{i} - \mu_{i} \, \longrightarrow \, 0, \quad i\to I.
    \]
    in $C(X,\RUCB(G)^*)$ with the topology of uniform convergence (see Lemma~\ref{lemma:witnesses.in.function.spaces}). Analogously, since $\nu_j$ is a witness of amenability, we get
    \[
        \left(h\otimes h^{-1}\right)\nu_{j} - \nu_{j} \, \longrightarrow \, 0, \quad j\to J
    \]
    in $C(\Samuel (G_{\Rsh}),\RUCB(H)^*)$. Since each $\mu_i$ has uniformly finite support, for each $i\in I$ there exists $j(i)\in J$ such that
    \[
        \left(\left(h\otimes h^{-1}\right)\nu_{j(i)}-\nu_{j(i)}\right)\circ \mu_i \to 0,\quad i\to I
    \]
    and
    \[
        \nu_{j(i)}\circ\left(\left(h\otimes h^{-1}\right)\!\mu_{i} - \mu_{i}\right) \, \longrightarrow \, 0, \quad i\to I
    \]
    in $C(X,\RUCB(H)^*)$. Altogether we get
    \begin{multline*}
     \left(h\otimes h^{-1}\right)(\nu_{j(i)}\circ \mu_i) - \nu_{j(i)}\circ \mu_i \\=     \left(\left(h\otimes h^{-1}\right)\nu_{j(i)}-\nu_{j(i)}\right)\circ \mu_i + \nu_{j(i)}\circ(\left(h\otimes h^{-1}\right)\!\mu_{i} - \mu_{i})\, \longrightarrow \, 0, \quad i\to I,    
    \end{multline*}
    as desired. 
    \end{proof}

\begin{definition} Let $G$ be any topological group. A topological subgroup $H \leq G$ is said to be \emph{well-placed} in $G$ if there exists an $H$-left-equivariant, uniformly continuous map from $G_{\Rsh}$ to $\Meanu(H_{\Rsh})_{\UEB}$. \end{definition} 

\begin{theorem}\label{theorem:amenable.actions.for.well.placed.subgroups}
Let $H$ be a well-placed topological subgroup of a topological group $G$. If a continuous action $G \curvearrowright X$ on a compact space $X$ is amenable, then the induced action $H \curvearrowright X$ is amenable, too.    
\end{theorem}

\begin{proof} Let us start off with some preparatory observations. For every function $f \in \UCB(H_{\Rsh})$, we let \begin{displaymath}
	f' \colon \, \Meanu(H_{\Rsh}) \, \longrightarrow \, \R, \quad \xi \, \longmapsto \, \xi(f) .
\end{displaymath} By definition of the UEB uniformity on $\Meanu(H_{\Rsh})$, we have \begin{equation}\label{double.ueb}
	\forall B \in \UEB(H_{\Rsh}) \colon \quad \{ f' \mid f \in B \} \in \UEB(\Meanu(H_{\Rsh})_{\UEB}) .
\end{equation} Moreover, for all $h \in H$, $f \in \UCB(H_{\Rsh})$ and $\xi \in \Meanu(H_{\Rsh})$, note that \begin{equation}\label{dual.action}
	(f \circ {\lambda_{h}})'(\xi) \, = \, \xi(f \circ {\lambda_{h}}) \, = \, (h\xi)(f) \, = \, f'(h\xi) .
\end{equation} Since $H$ is well-placed in $G$, there exists an $H$-left-equivariant, uniformly continuous map $\varphi \colon G_{\Rsh} \to \Meanu(H_{\Rsh})_{\UEB}$. From~\eqref{double.ueb} and uniform continuity of $\varphi$, we infer that \begin{equation}\label{composition.ueb}
	\forall B \in \UEB(H_{\Rsh}) \colon \quad \{ {f'} \circ \varphi \mid f \in B \} \in \UEB(G_{\Rsh}) .
\end{equation} Thanks to equivariance of $\varphi$, if $h \in H$ and $f \in \UCB(H_{\Rsh})$, then \begin{displaymath}
	({(f \circ {\lambda_{h}})'} \circ \varphi)(g) = (f \circ {\lambda_{h}})'(\varphi(g)) \stackrel{\eqref{dual.action}}{=} f'(h\phi(g)) = f'(\phi(hg)) = ({f'} \circ \varphi \circ {\lambda_{h}})(g)
\end{displaymath} for all $g \in G$. That is, \begin{equation}\label{crucial.equivariance}
	\forall h \in H \, \forall f \in \UCB(H_{\Rsh}) \colon \quad {(f \circ {\lambda_{h}})'} \circ \varphi = {f'} \circ \varphi \circ {\lambda_{h}} .
\end{equation}

Suppose now that $\mu_{\iota} \colon X \to \Delta(G)$ $(\iota \in I)$ is a witness of amenability for a continuous action $G \curvearrowright X$ on a compact space $X$. For each $\iota \in I$, consider the continuous map \begin{displaymath}
	\nu_{\iota} \colon \, X \, \longrightarrow \, \Meanu(H_{\Rsh}), \quad x \, \longmapsto \, \sum\nolimits_{g \in \spt \mu_{\iota}(x)} \mu_{\iota}(x)(g)\varphi(g) 
\end{displaymath} and note that \begin{equation}\label{rewriting}
	\forall x \in X \, \forall f \in \UCB(H_{\Rsh}) \colon \quad \nu_{\iota}(x)(f) = \mu_{\iota}(x)({f'} \circ \varphi) .
\end{equation} Consequently, if $h \in H$ and $B \in \UEB(H_{\Rsh})$, then $\{ {f'} \circ \varphi \mid f \in B \} \in \UEB(G_{\Rsh})$ by~\eqref{composition.ueb} and therefore \begin{align*}
	&\sup\nolimits_{x \in X} \sup\nolimits_{f \in B} \lvert (h\nu_{\iota}(x))(f) - \nu_{\iota}(hx)(f) \rvert \\
	& \qquad = \, \sup\nolimits_{x \in X} \sup\nolimits_{f \in B} \lvert \nu_{\iota}(x)(f \circ {\lambda_{h}}) - \nu_{\iota}(hx)(f) \rvert \\
	& \qquad \stackrel{\eqref{rewriting}}{=} \, \sup\nolimits_{x \in X} \sup\nolimits_{f \in B} \lvert \mu_{\iota}(x)({(f \circ {\lambda_{h}})'} \circ \varphi) - \mu_{\iota}(hx)({f'} \circ \varphi) \rvert \\
	& \qquad \stackrel{\eqref{crucial.equivariance}}{=} \, \sup\nolimits_{x \in X} \sup\nolimits_{f \in B} \lvert \mu_{\iota}(x)({f'} \circ \varphi \circ {\lambda_{h}}) - \mu_{\iota}(hx)({f'} \circ \varphi) \rvert \\
	& \qquad = \, \sup\nolimits_{x \in X} \sup\nolimits_{f \in B} \lvert (h\mu_{\iota}(x))({f'} \circ \varphi) - \mu_{\iota}(hx)({f'} \circ \varphi) \rvert \\
	& \qquad \longrightarrow \, 0 \quad (\iota \to I) .
\end{align*} Hence, $(\nu_{\iota})_{\iota \in I}$ is a witness of amenability for the action $H \curvearrowright X$, as desired. \end{proof}

\begin{corollary}\label{corollary:well-placed.subgroups.inherit} Every well-placed topological subgroup of an exact (resp., amenable) topological group is exact (resp., amenable). \end{corollary}

Here is a natural example.

\begin{proposition}\label{proposition:open.subgroups} Any open subgroup of a topological group is well-placed in the latter. In particular, in a discrete group any subgroup is well-placed. \end{proposition}

\begin{proof} Let $H$ be an open subgroup of a topological group $G$. Choose any section $\sigma$ for the quotient map $G \to \{ Hx \mid x \in G \}, \, x \mapsto Hx$. Then the map \begin{displaymath}
    \varphi \colon \, G \, \longrightarrow \, H, \quad x \, \longmapsto \, x\sigma(Hx)^{-1}
\end{displaymath} is $H$-left-equivariant: if $h \in H$, then \begin{displaymath}
    \varphi(hx) \, = \, hx\sigma(Hhx)^{-1} \, = \, hx\sigma(Hx)^{-1} \, = \, h\varphi(x)
\end{displaymath} for all $x \in G$. To see that $\varphi$ is also right-uniformly continuous, let $U \in \Neigh(H)$. Since $H$ is open in $G$, we have $U \in \Neigh(G)$. Now, for all $x,y \in G$ with $xy^{-1} \in U$, we see that $xy^{-1} \in H$ and so $Hx=Hy$, wherefore \begin{displaymath}
    \varphi(x)\varphi(y)^{-1} \, = \, x\sigma(Hx)\sigma(Hy)^{-1}y^{-1} \, = \, xy^{-1} \, \in \, U .
\end{displaymath} This proves right-uniform continuity of $\varphi$, hence completing the argument. \end{proof}

Note that well-placedness does not imply amenable embeddability: while any subgroup of a discrete group is well-placed by Proposition~\ref{proposition:open.subgroups}, non-exact groups are not amenably embedded in any supergroup. It would interesting to know if an exact discrete group is amenably embedded in all of its supergroups.

\vspace{0.2cm}

We now turn to the left uniformity.

\begin{definition} Let $G$ be any topological group. A topological subgroup $H \leq G$ is said to be \emph{skew-well-placed} in $G$ if there exists an $H$-left-equivariant, uniformly continuous map from $G_{\Lsh}$ to $\Meanu(H_{\Lsh})_{\UEB}$. \end{definition} 

\begin{proposition}\label{proposition:cocompact lattices}
    Let $G$ be a locally compact group and $\Gamma < G$ a cocompact lattice. Then $\Gamma$ is skew-well-placed in $G$.
    \end{proposition}
    \begin{proof}
    
    Since $\Gamma$ is cocompact, there exists a compact subset $K\subseteq G$ such that $\bigcup_{h\in\Gamma} hK = G$. Let $\theta\in C_c(G,[0,1])$ be a function with compact support such that $\theta|_K$ is positive. Then, for each $x\in G$ there exists an $h\in \Gamma$ such that $(\lambda_h \theta)(x) > 0$. We now claim that there exists an $N > 0$ such that for every $x\in G$ there are at most $N$ elements of $\Gamma$ with $(\lambda_h \theta)(x) > 0$. Indeed, denoting $K'$ the support of $\theta$ we observe that $(K')^2$ is compact; since $\Gamma$ is discrete, it contains at most $N$ elements of $\Gamma$. The condition we $(\lambda_h \theta)(x) > 0$ is equivalent to $x\in hK'$, and if $x\in hK'\cap h'K'$, we deduce that $(h')^{-1}h\in (K')^2$; this proves the claim. 
    
    Thus, for $h \in \Gamma$ and $x\in G$, we can define
    \[
    \psi_h(x) \, \defeq \, \frac{(\lambda_h \theta)(x)}{\sum_{h\in \Gamma}(\lambda_h \theta)(x)}.
    \]
    This ensures that the family $(\psi_h)_{h\in\Gamma}$ forms a uniformly finite partition of unity on $G$. Moreover, since $\psi_h = \lambda_h \psi_1$ for all $h\in \Gamma$, we observe that this partition of unity is $\Gamma$-left-equivariant: $\lambda_{h}\psi_{h'} = \psi_{hh'}$ for all $h,h'\in\Gamma$.
    
We now can define the map \begin{displaymath}
    \varphi\colon \, G \, \longrightarrow \, P(\Gamma), \quad x \,\longmapsto \, \sum\nolimits_{h\in\Gamma} \psi_h(x)\delta_h .
\end{displaymath} It is $\Gamma$-left-equivariant since for all $h\in\Gamma$ we have \begin{align*}
    \varphi(hx) \, &= \, \sum\nolimits_{h'\in \Gamma}\psi_{h'}(hx)\delta_{h'} \, = \, \sum\nolimits_{h'\in \Gamma}\psi_{h^{-1}h'}(x)\delta_{h'} \\
    &= \, \sum\nolimits_{h'\in \Gamma}\psi_{h'}(x)\delta_{hh'} \, = \, h\cdot \varphi(x).
\end{align*} Let us now check that the map $\varphi$ is left-uniformly continuous. Indeed, the function $\psi_1$ is compactly supported and therefore left-uniformly continuous, that is, for every $\eps > 0$ there is a neighbourhood $U$ of the identity such that for all $u\in U$ we have $\sup_{x\in G}|\psi_1(xu)-\psi_1(x)| < \eps$. It follows that for every $h\in \Gamma$ we get $\sup_{x\in G}|\psi_h(xu)-\psi_h(x)| < \eps$.
    
    Since as observed above, the sum
    \[
        \sum\nolimits_{h\in\Gamma} \psi_h(x)\delta_h
    \]
    has at most $N$ summands for every $x\in G$, we get 
    \[
    \sup\nolimits_{x\in G}\norm{\varphi(xu)-\varphi(x)}_1 \, \leq \, N \sup\nolimits_{x\in G} |\psi_h(xu) - \psi_h(x)| \, < \, N\eps,
    \]
    as desired.
    \end{proof}

\begin{theorem}\label{theorem:skew.amenable.actions.for.skew.well.placed.subgroups}
Let $H$ be a skew-well-placed topological subgroup of a topological group $G$. If a continuous action $G \curvearrowright X$ on a compact space $X$ is skew-amenable, then the induced action $H \curvearrowright X$ is skew-amenable, too. \end{theorem}

\begin{proof} The proof is completely analogous to the one of Theorem~\ref{theorem:amenable.actions.for.well.placed.subgroups}. \end{proof}

\begin{corollary}\label{corollary:skew.exactness} A skew-well-placed topological subgroup of a skew-exact (resp., skew-amenable) topological group is skew-exact (resp., skew-amenable). \end{corollary}

\begin{question} Is every closed topological subgroup of a locally compact group well-placed resp. skew-well-placed? \end{question}

\section{Exactness of Polish groups with metrizable UMF}
\label{sec:UMF}
In this section, we show that, if a Polish group $G$ has a metrizable universal minimal flow, then the action of $G$ on the latter is amenable. 

We start off by clarifying some notation and terminology. Given a uniform space $X$, we let $\widehat{X}$ denote the \emph{Hausdorff completion} of $X$ (see~\cite[II, \S3.7]{bourbaki} or~\cite{robertson} for details), and we recall that precompactness of $X$ is equivalent to compactness of $\widehat{X}$. Now, let $H$ be a subgroup of a topological group $G$. In the following, the set $G/H = \{ xH \mid x \in G\}$ will be endowed with the \emph{right uniformity} \begin{displaymath}
	\left\{ E \subseteq (G/H) \times (G/H) \left\vert \, \exists U \in \Neigh(G) \, \forall x,y \in G \colon \, xy^{-1}\! \in U \Longrightarrow \, (xH,yH) \in E \right\} \! , \right.\! 
\end{displaymath} which is the finest uniformity on $G/H$ such that \begin{displaymath}
	G \, \longrightarrow \, G/H, \quad x \, \longmapsto \, xH
\end{displaymath} is uniformly continuous with respect to the right uniformity on $G$. We say that the subgroup $H$ is \emph{co-precompact} in $G$ if the uniform space $G/H$ is precompact, which means precisely that \begin{displaymath}
	\forall U \in \Neigh(G) \ \exists C \in \Pfin(G) \colon \qquad G = UCH .
\end{displaymath} If $H$ is precompact in $G$, then the action of of $G$ by by left translations on $G/H$ naturally induces a continuous action of the topological group $G$ on the compact space $\widehat{G/H}$.

\begin{theorem}\label{theorem:main} Let $G$ be a topological group and let $H$ be an amenable topological subgroup of $G$ such that $G/H$ is precompact. Then $G \curvearrowright \widehat{G/H}$ is amenable. \end{theorem}

\begin{proof} By a standard completion argument, it suffices to show that, for all $\varepsilon \in \R_{>0}$, $E \in \Pfin(G)$, and $B \in \RUEB(G)$, there exists a uniformly continuous map $\varphi \colon G/H \to \Delta(G)$ with relatively compact range such that \begin{displaymath}
	\forall g \in E \ \forall x \in G \ \forall f \in B \colon \qquad \lvert \langle g \varphi(xH),f\rangle - \langle \varphi(gxH), f \rangle \rvert \leq \varepsilon .
\end{displaymath} For this, let $\varepsilon \in \R_{>0}$, $E \in \Pfin(G)$, and $B \in \RUEB(G)$. Without loss of generality, we may and will assume that $e \in E$. As $B \in \RUEB(G)$, there is $U \in \Neigh(G)$ such that $U^{-1} = U$ and \begin{equation}\label{ueb}
	\forall f \in B \ \forall u \in U \colon \qquad \lVert f - (f \circ {\lambda_{u}}) \rVert_{\infty} \leq \tfrac{\varepsilon}{2} .
\end{equation} Pick any $V \in \Neigh(G)$ with $V^{-1} = V$ and $VV \subseteq U$. Clearly, \begin{displaymath}
    W \, \defeq \, \bigcap\nolimits_{g \in E} g^{-1}Vg \, \in \, \Neigh(G)
\end{displaymath} and $W^{-1} = W$. Fix any $W_{0} \in \Neigh(G)$ with $W_{0}^{-1} = W_{0}$ and $W_{0}W_{0} \subseteq W$. Since $G/H$ is precompact, there exists a finite subset $C \subseteq G$ such that $G = W_{0}CH$. As both $C$ and $E$ are finite, there exists a finite subset $H_{0} \subseteq H$ such that \begin{equation}\label{distance}
	\forall c,c' \in C \ \forall g \in E \colon \qquad \bigl(gc \in Uc'H \ \Longrightarrow \ \exists h \in H_{0}\colon \, gc \in Uc'h\bigr) .
\end{equation} Note that $\{ f \circ {\lambda_{c}} \mid f \in B, \, c \in C \} \in \RUEB(G)$. According to~\cite[Theorem~3.2]{MR3809992}, amenability of the topological group $H$ thus implies the existence of $\mu \in \Delta(H)$ such that \begin{equation}\label{folner}
	\forall f \in B \ \forall c \in C \ \forall h \in H_{0} \colon \qquad \lvert \langle ch\mu, f\rangle - \langle c\mu,f \rangle \rvert \leq \tfrac{\varepsilon}{2} .
\end{equation} We observe that $(\{ wcH \mid w \in W \})_{c \in C}$ constitutes a finite uniform cover of $G/H$: indeed, if $x \in G$, then there exists $c \in C$ with $x \in W_{0}cH$, whence \begin{displaymath}
	\{ wxH \mid w \in W_{0} \} \, \subseteq \, \{ wcH \mid w \in W \} .
\end{displaymath} Therefore, by~\cite[IV.11, Theorem, p.~62]{isbell}, there exists a uniformly continuous partition of unity \begin{displaymath}
	\psi_{c} \colon \, G/H \, \longrightarrow \, [0,1] \qquad (c \in C)
\end{displaymath} subordinate to $(\{ wcH \mid w \in W \})_{c \in C}$. Since $(\psi_{c})_{c \in C}$ is a uniformly continuous partition of unity, the map \begin{displaymath}
	\varphi \colon \, G/H \, \longrightarrow \, \Delta(G), \quad z \, \longmapsto \, \sum\nolimits_{c \in C} \psi_{c}(z)\cdot c\mu 
\end{displaymath} is well defined and uniformly continuous. Moreover, the range of $\varphi$ is contained in the compact subset \begin{displaymath}
	\Delta (C \spt(\mu)) \, \subseteq \, \Delta(G) .
\end{displaymath} It remains to show that \begin{equation}\label{claim}
	\forall g \in E \ \forall x \in G \ \forall f \in B \colon \qquad \lvert \langle g \varphi(xH),f \rangle - \langle \varphi(gxH),f \rangle \rvert \leq \varepsilon .
\end{equation} To this end, let $g \in E$. We first verify that \begin{equation}\label{step}
    \begin{split}
	&\forall x \in G \ \forall c,c' \in C \colon \\
    & \qquad \psi_{c}(xH)\psi_{c'}(gxH) > 0 \ \Longrightarrow \ \sup\nolimits_{f \in B} \lvert \langle gc\mu,f \rangle - \langle c'\mu,f \rangle \rvert \leq \varepsilon .
    \end{split}
\end{equation} Let $x \in G$ and $c,c' \in C$ with $\psi_{c}(xH)\psi_{c'}(gxH) > 0$. Then $x \in WcH$ and $gx \in Wc'H$, from which we infer that \begin{displaymath}
	c \, \in \, WxH \, = \, g^{-1}gWg^{-1}gxH \, \subseteq \, g^{-1}VgxH \, \subseteq \, g^{-1}VWc'H \, \subseteq \, g^{-1}Uc'H ,
\end{displaymath} that is, $gc \in Uc'H$. Due to~\eqref{distance}, thus there exists $h \in H_{0}$ with $gc \in Uc'h$. Consequently, for every $f \in B$, \begin{align*}
	&\lvert \langle gc\mu,f \rangle - \langle c'\mu,f \rangle \rvert \, \leq \, \lvert \langle gc\mu,f \rangle - \langle c'h\mu,f \rangle \rvert + \lvert \langle c'h\mu,f \rangle - \langle c'\mu,f \rangle \rvert \\
    &\qquad \leq \sum\nolimits_{x \in H} \mu(x)\lvert f(gcx) - f(c'hx) \rvert + \lvert \langle c'h\mu,f \rangle - \langle c'\mu,f \rangle \rvert \\
    &\qquad \stackrel{\eqref{ueb}+\eqref{folner}}{\leq} \, \tfrac{\varepsilon}{2} + \tfrac{\varepsilon}{2} \, = \, \varepsilon .
\end{align*} This proves~\eqref{step}. We conclude that, for every $x \in G$ and every $f \in B$, \begin{align*}
	&\lvert \langle g\varphi(xH),f \rangle - \langle \varphi(gxH),f \rangle \rvert \\
    & \qquad \qquad \leq \, \sum\nolimits_{c \in C} \sum\nolimits_{c' \in C} \psi_{c}(xH)\psi_{c'}(gxH) \lvert \langle gc\mu,f \rangle - \langle c'\mu,f \rangle \rvert \, \stackrel{\eqref{step}}{\leq} \, \varepsilon .
\end{align*} This proves~\eqref{claim} and hence completes the argument. \end{proof}

If $G$ is a topological group, then $\UMF(G)$ denotes its universal minimal flow.

\begin{corollary}\label{corollary:main} If $G$ is a Polish group and $\UMF(G)$ is metrizable, then $G \curvearrowright \UMF(G)$ is amenable. \end{corollary}

\begin{proof} This follows by~\cite[Theorem~1.1]{BMT} and Theorem~\ref{theorem:main}. \end{proof}

The study of Polish groups with metrizable UMF was initiated in~\cite{GlasnerWeiss02,KPT}. Examples of non-amenable Polish groups with metrizable UMF include \begin{itemize}
    \item[$\bullet$\,] the homeomorphism group of the Cantor space (for non-amenability see~\cite[p.~70]{KechrisSokic}, for a description of the UMF see~\cite{GlasnerWeiss03}), and
    \item[$\bullet$\,] the automorphism group of the random distributive lattice (see~\cite[Theorems~0.1 and~0.2]{KechrisSokic}).
\end{itemize} While containment of a co-precompact amenable topological subgroup does not imply amenability of the ambient topological group, it does imply exactness of the latter according to Theorem~\ref{theorem:main}. In particular, Polish groups with metrizable UMF are exact by Corollary~\ref{corollary:main}.

\section{Products and inverse limits}

\subsection{Products of actions}

The aim of this subsection is to prove the following proposition.

\begin{proposition}\label{prop: exacproduct}
    Let $G_1 \curvearrowright X_1$ and $G_2 \curvearrowright X_2$ be actions of topological groups on compact spaces. If both actions are amenable, then the product action $G_1 \times G_2 \curvearrowright X_1 \times X_2$ is amenable. 
\end{proposition}

There is a somewhat subtle notion of tensor product of uniform measures, see \cite[$\S$9]{PachlBook} for more details. In our setting, Lemma \ref{lemma:cut.off} allows us to work with witnesses of amenability that lie in $\Delta(G_i)$. For those measure all the more delicate properties of the tensor product such as Fubini's theorem are obvious. The key observation in order to prove Proposition \ref{prop: exacproduct} is the following lemma:

\begin{lemma}\label{lem tensor map is continuous}
Let $G_1,G_2$ be topological groups. 
For $B\in \RUEB(G_1\times G_2)$, we have 
$B_1:=\{f(\cdot,g_2)\mid f\in B,\,g_2\in G_2\} \in \RUEB(G_1)$ and, similarly,
     $B_2:=\{f(g_1,\cdot)\mid f\in B,\,g_1\in G_2\} \in \RUEB(G_2)$.
\end{lemma}
\begin{proof}
Let $\varepsilon>0$. Then there exist $U_i\in \Neigh(G_i)\,(i=1,2)$ such that for all $f\in B$, $g_1,h_1\in G_1$  and $g_2,h_2\in G_2$, the following implication hold: \[g_1h_1^{-1}\in U_1,\,\,g_2h_2^{-1}\in U_2\implies |f(g_1,g_2)-f(h_1,h_2)|<\varepsilon.\]
In particular, if $g_1,h_1\in G_1$ satisfy $g_1h_1^{-1}\in U_1$, then for every $g_2\in G_2$, $|f(g_1,g_2)-f(h_1,g_2)|<\varepsilon$. This shows that $B_i$ belongs to  $\RUEB(G_i)$ for $i \in \{1,2\}$.
\end{proof}
\begin{proof}[Proof of Proposition \ref{prop: exacproduct}]
Let $\mu_{i,\iota}\colon X_i\to \Delta(G_i)$ be a witness of amenability of the $G_i$ action on $X_i$ for $i=1,2$. Define for each $\iota\in I$ a map 
$\mu_{\iota}\colon X_1 \times X_2\to \Delta(G_1 \times G_2)$ by 
\[\mu_{\iota}(x_1,x_2)=\mu_{1,\iota}(x_1) {\otimes} \mu_{2,\iota}(x_2)\in \Delta(G_1 \times G_2).\]

From the properties of $\mu_{i,\iota} \colon X_i \to \Delta(G_i)$ guaranteed in Lemma \ref{lemma:cut.off}, it is easy to see that $\mu_{\iota} \colon X_1 \times X_2 \to \Delta(G_1 \times G_2)$ is continuous.
 
We show that $(\mu_{\iota})_{\iota\in I}$ is a witness for the amenability of the action $G_1 \times G_2 \curvearrowright X_1 \times X_2$. Let $B\in \RUEB(G)$ and $g^o=(g^o_1,g^o_2)\in G$ be given. 
Since $(\mu_{1,\iota})_{\iota\in I}$ is a witness for the amenability of the action $G_1\curvearrowright X_1$ and the set $B_1$ is in $\RUEB(G_1)$, it follows that the net of functions on $G_2$ given by 
    \[g_2 \mapsto \sup_{x_1\in X_1}\sup_{f\in B}| g_1^o \mu_{1,\iota}(x_1)(f(\cdot,g_2)) -  \mu_{1,\iota}(g_1^o x_1)(f(\cdot,g_2)) |\]
    converges to $0$ uniformly on $G_2$. Therefore, as uniform means have norm one, we can integrate with respect to $\mu_{2,\iota}(g_2^o x_2)$ and obtain that the net
    \eqa{
        \sup_{(x_1,x_2)\in X_1\times X_2}\sup_{f\in B}| (g^o_1,1)\mu_{\iota}(x_1,g^o_2x_2)(f) - \mu_{\iota}(g^o_1x_1,g^o_2x_2)(f)|
    }
    converges to $0$ as $\iota \to I$. By symmetry, we also obtain that
    \eqa{
        \sup_{(x_1,x_2)\in X_1\times X_2}\sup_{f\in B}| (1,g^o_2)\mu_{\iota}(x_1,x_2)(f) - \mu_{\iota}(x_1,g^o_2x_2)(f)|
    }
converges uniformly to $0$ as $\iota \to I.$ Multiplying the measures in this expression by $(g_1^o,1)$ and putting all the above together, we have shown that 
    \[\lim_{\iota \to I}\sup_{(x_1,x_2)\in X_1\times X_2}\sup_{f\in B}|g^o\mu_{\iota}(x)(f)-\mu_{\iota}(g^ox)(f)|=0\]
as required. Therefore, the action $G_1 \times G_2 \curvearrowright X_1 \times X_2$ is amenable.
\end{proof} 
\subsection{Inverse limits}
The aim of this section is to show that an inverse limit of exact groups is again exact. As a corollary we obtain Kirchberg's theorem stating that a locally compact group $G$ is exact if $G/G^{\circ}$ is compact, where $G^{\circ}$ denotes the connected component of the neutral element in $G$.

Let $(I,\leq)$ be a poset, let $(G_i)_{i \in I}$ be a family of topological groups and $f_{ij} \colon G_j \to G_i$ be continuous homomorphisms for $i \leq j$, such that $f_{ik} = f_{ij} \circ f_{jk}$  whenever $i \leq j \leq k$. In this situation, we can define the inverse limit
$$G:= \lim_{(I,\leq)} G_i := \left\{ (a_i)_i \in \prod_{i \in I} G_i \mid a_i = f_{ij}(a_j) \ \forall i \leq j \right\}.$$
By definition $G$ is a closed subgroup the infinite product $\prod_{i \in I} G_i$ and inherits a natural group topology as such. We will assume throughout this section that the natural projections $f_i \colon G \to G_i$ have dense image. This is a subtle assumption, indeed even if all $f_{ij}$ are surjective, the inverse limit can be trivial in general.

\begin{lemma} \label{lem:martin}
Consider the inverse limit $G = \lim_{i \in I} G_{i}$ of an inverse system of topological groups $(G_{i})_{i \in I}$. For each $i \in I$, let $\pi_{i} \colon G \to G_{i}$ denote the corresponding projection. Then, for every $B \in \RUEB(G)$ and every $\varepsilon \in \R_{>0}$, there exist $i \in I$ and $B' \in \RUEB(G_{i})$ such that \begin{displaymath}
    \forall f \in B \ \exists f' \in B' \colon \quad \Vert f-(f' \circ \pi_{i}) \Vert_{\infty} \, \leq \, \varepsilon.
\end{displaymath} \end{lemma}

\begin{proof} As $B$ is right-uniformly equicontinuous, there exists an identity neighborhood $U$ in $G$ such that \begin{displaymath}
	\forall x, y \in G\colon \quad xy^{-1} \in U \, \Longrightarrow \ \lvert f(x) - f(y) \rvert < \varepsilon.
\end{displaymath} By definition of the topology on $G$, there exist $i \in I$ and an identity neighborhood $U_{i}$ in $G_{i}$ such that $\pi_{i}^{-1} (U_{i}) \subseteq U$. Since $B$ is norm-bounded, we see that $r \defeq \sup\nolimits_{f \in B} \Vert f \Vert_{\infty} < \infty$. Now, we find a right-invariant continuous pseudo-metric $d_{i}$ on~$G_{i}$ with $B_{d_{i}}(e_{i},2r) \subseteq U_{i}$: indeed, by Urysohn's lemma for uniform spaces (see, e.g.,~\cite[pp.~182--183]{MR884154}), there exists a right-uniformly continuous function $F_{i} \colon G_{i} \to [0,3r]$ with $F_{i}(e_{i}) = 0$ and $F_{i}^{-1}([0,2r]) \subseteq U_{i}$, wherefore \begin{displaymath}
	d_{i} \colon \, G_{i} \times G_{i} \, \longrightarrow \, \R_{\geq 0}, \quad (x,y) \, \longmapsto \, \sup\{ \vert F_{i}(xg) - F_{i}(yg) \vert \mid g \in G_{i} \}
\end{displaymath} is a continuous right-invariant pseudo-metric on $G_{i}$ such that $B_{d_{i}}(e_{i},2r) \subseteq U_{i}$. Note that $B' \defeq \Lip_{1}(G_{i},d_{i};[-r,r]) \in \RUEB(G_{i})$. We will show that \begin{displaymath}
	\forall f \in B \ \exists f' \in B' \colon \quad \Vert f-(f' \circ \pi_{i}) \Vert_{\infty} \, \leq \, \varepsilon.
\end{displaymath} To this end, let $f \in B$. If $x,y \in G$, then either $xy^{-1} \in U$ and so $\lvert f(x) - f(y) \rvert < \varepsilon$, or $xy^{-1} \not\in U \supseteq \pi_{i}^{-1} (U_{i}) \supseteq \pi_{i}^{-1} (B_{d_{i}}(e_{i},2r))$ and therefore \begin{align*}
    d_{i}(\pi_{i}(x),\pi_{i}(y)) \, = \, d_{i}\!\left(\pi_{i}\!\left(xy^{-1}\right)\!,e_{i}\right) \, \geq \, 2r \, \geq \, 2\lVert f\rVert_{\infty} \, \geq \, \vert f(x) - f(y) \vert .
\end{align*} Hence, for all $x,y \in G$, \begin{displaymath}
	\lvert f(x) - f(y) \rvert \, \leq \, \max\{ d_{i}(\pi_{i}(x),\pi_{i}(y)), \varepsilon \} \, \leq \, d_{i}(\pi_{i}(x),\pi_{i}(y)) + \varepsilon .
\end{displaymath} In turn, \cite[Lemma~5.2]{MR4152622} asserts that the existence of $f' \in \Lip_{1}(G_{i},d_{i}) \cap \ell^{\infty}(G_{i})$ such that $\lVert f - ({f'} \circ {\pi_{i}}) \rVert_{\infty} \leq \varepsilon$. It now follows that $f'' \defeq (f' \wedge r) \vee (-r) \in B'$ and $\lVert f - ({f''} \circ {\pi_{i}}) \rVert_{\infty} \leq \varepsilon$, as desired. \end{proof}

\begin{lemma} \label{lem:inverse}
If $G_i \curvearrowright X_i$ is an amenable action on a compact space for all $i \in I$, then the action $G \curvearrowright \prod_{i \in I} X_i$ is amenable. In particular, if $G_i$ is exact for all $i \in I$, then $G$ is exact.
\end{lemma}
\begin{proof} Clearly $X: = \prod_{i \in I} X_i$ is compact and $G \curvearrowright X$ is continuous. We need to provide a net of continuous maps 
$\mu_{\iota} \colon X \to \Meanu(G_{\Rsh})$ $(\iota \in I)$ such that, for every $g \in G$ and every $B \in \RUEB(G)$, \begin{displaymath}
    \sup\nolimits_{x \in X} \sup\nolimits_{f \in B} \vert (g\mu_{\iota}(x))(f) - \mu_{\iota}(gx)(f) \vert \, \longrightarrow \, 0 \quad (\iota \to I) .
\end{displaymath}
By Lemma \ref{lemma:cut.off}, for each $i \in I$, there exist nets of continuous maps 
$\mu_{\iota} \colon X_i \to \Delta(G_{i,\Rsh})$ $(\iota \in I_i)$ such that, for every $g \in G_i$ and every $B \in \RUEB(G_i)$, \begin{displaymath}
    \sup\nolimits_{x \in X_i} \sup\nolimits_{f \in B} \vert (g\mu_{\iota}(x))(f) - \mu_{\iota}(gx)(f) \vert \, \longrightarrow \, 0 \quad (\iota \to I_i),
\end{displaymath}
such that $\cup_{x \in X_i} {\rm spt}(\mu_{\iota}(x)) \subseteq G_i$ is finite for every $\iota \in I_i.$

Since the image $f_i \colon G \to G_i$ is dense we may further consider an approximate section $s_{i,U} \colon G_i \to G$, i.e.\ a map, such that $(f_i \circ s_{i,U})(g) g^{-1} \in U$, for all $g \in G$ and some fixed neighborhood $U$ of $1 \in G_i.$

We consider now the composition
\begin{center}
\begin{tikzpicture}
    \node (A) at (0,0) {$X$};
    \node (B) at (2,0) {$X_i$};
    \node (C) at (4,0) {$\Delta(G_{i,\Rsh})$};
    \node (D) at (7,0) {$\Delta(G_{\Rsh}).$};

    \draw[->] (A) -- (B) node[midway, above] {$\pi_i$};
    \draw[->] (B) -- (C) node[midway, above] {$\mu_{i,\iota}$};
    \draw[->] (C) -- (D) node[midway, above] {$\Delta(s_{i,U})$};

    \draw[->, bend right=20] (A) to node[midway, below] {$\mu_{i,\iota,U}$} (D);
\end{tikzpicture}
\end{center}
and claim that they provide a witness of amenability for the action $G \curvearrowright X.$ Indeed, this follows easily by Lemma \ref{lem:martin}.
\end{proof}

\begin{corollary}
Arbitrary products of amenable actions are amenable. In particular, an arbitrary product of exact groups is exact.
\end{corollary}
\begin{proof} The finite case follows by induction from Proposition \ref{prop: exacproduct}. The infinite case is then a consequence of Lemma \ref{lem:inverse}.
\end{proof}

\begin{corollary}[Kirchberg]\label{corollary:kirchberg} Let $G$ be a locally compact group such that $G/G^{\circ}$ is compact. Then, $G$ is exact.
\end{corollary}
\begin{proof}
By the Gleason-Yamabe theorem, any such $G$ is an inverse limit of virtually connected Lie groups. Since any virtually connected Lie group admits an amenable closed subgroup such that the quotient space is compact, we conclude that $G$ is an inverse limit of exact groups by Theorem~\ref{theorem:main}. Hence, $G$ is itself exact according to Lemma~\ref{lem:inverse}.
\end{proof}

Affirmative answers to the following two questions, the first of which has been suggested to the authors by the anonymous referee, would provide a shorter proof of Corollary~\ref{corollary:kirchberg}.

\begin{question}\label{question:extension} Let $N$ be a normal subgroup of a topological group~$G$. \begin{enumerate}
    \item[$(1)$] If $N$ is compact and $G/N$ is exact, is then $G$ exact, too?
    \item[$(2)$] If $N$ is exact and $G/N$ is compact, is then $G$ exact, too? 
\end{enumerate} \end{question}

\section{Further examples and applications}

\subsection{Exactness of unitary groups}

In this section, we study exactness of unitary groups of Hilbert spaces and show that they are non-exact if the dimension is infinite. Our analysis relies on Gromov's work providing non-exact discrete groups \cite{Gromov00,Gromov03} and we were unable to find a more direct proof of our non-exactness result. It would be desirable to find a more direct argument. For comparison, the reader is referred to~\cite{DeLaHarpe73,GiordanoPestov} for subtleties concerning the amenability of topological groups of unitaries.

Given a Hilbert space $H$ and a subgroup $G$ of its unitary group $\UU(H)$, we let $G_{\mathrm{sot}}$ denote the topological group obtained by endowing $G$ with the corresponding strong operator topology, and we let $G_{\mathrm{norm}}$ denote the topological group obtained by equipping $G$ with the norm topology.

\begin{theorem}\label{theorem:skew.exactness.for.unitary.groups} Let $\nu$ be a left Haar measure on a locally compact group~$G$. Consider $G$ as a topological subgroup of $\UU(H)_{\mathrm{sot}}$ for $H \defeq L^{2}(G,\nu)$, embedded via the left regular representation. Then $G$ is skew-well-placed in $\UU(H)_{\mathrm{sot}}$. \end{theorem}

\begin{proof} Consider the sphere $\mathbb{S} \defeq \{ x \in H \mid \Vert x \Vert = 1 \}$ endowed with the metric induced by the norm of $H$ and equipped with the natural isometric action by $\UU(H)$. Fix some $x_{0} \in \mathbb{S}$. The map \begin{displaymath}
    \varphi \colon \, \UU(H) \, \longrightarrow \, \mathbb{S}, \quad g \, \longmapsto \, gx_{0}
\end{displaymath} is $\UU(H)$-equivariant (with respect to the left action of $\UU(H)$ on itself) and left-uniformly continuous with respect to the strong operator topology on $\UU(H)$. Furthermore, the map \begin{displaymath}
    \psi \colon \, \mathbb{S} \, \longrightarrow \, \Prob(G,\nu), \quad f \, \longmapsto \, \vert f \vert^{2}
\end{displaymath} is clearly $G$-equivariant and, moreover, $2$-Lipschitz with respect to the metric induced by $\Vert \cdot \Vert_{\nu,1}$ on $\Prob(G,\nu)$: indeed, if $f,f' \in \mathbb{S}$, then \begin{align*}
	\Vert \psi(f) - \psi(f') \Vert_{\nu,1} \, & = \, \left\lVert \vert f \vert^{2} - \vert f' \vert^{2} \right\rVert_{\nu,1} \, = \, \int \left\lvert \vert f(x)\vert^{2} - \vert f'(x) \vert^{2} \right\rvert \, \mathrm{d}\nu(x) \\
	& = \, \int \left\lvert \vert f(x) \vert + \vert f'(x)\vert \right\rvert \cdot \left\lvert \vert f(x)\vert - \vert f'(x)\vert \right\rvert \, \mathrm{d}\nu(x) \\
	& \leq \, \int ( \vert f(x) \vert + \vert f'(x)\vert) \vert f(x) - f'(x) \vert \, \mathrm{d}\nu(x) \\
	& = \, \left\lvert \langle \vert f\vert + \vert f' \vert ,\vert f - f' \vert \rangle \right\rvert \\
    &\leq \, \left\lVert \vert f \vert + \vert f' \vert \right\rVert_{\nu,2}\cdot \left\lVert f - f' \right\rVert_{\nu,2} \\
    &\leq \, 2 \left\lVert f-f' \right\rVert_{\nu,2} .
\end{align*} Combining these observations with Remark~\ref{remark:haar.measure}, we conclude that the composition ${\Xi_{\nu}} \circ \psi \circ \varphi \colon \UU(H) \to \Meanu(G_{\Lsh})$ is a $G$-left-equivariant, uniformly continuous map from $(\UU(H)_{\mathrm{sot}})_{\Lsh}$ to $\Meanu(G_{\Lsh})$, as desired. \end{proof}

\begin{corollary}\label{corollary:skew.exactness.for.unitary.groups} Let $G$ be a group. Consider $G$ as a subgroup of the unitary group $\UU(H)$ of the Hilbert space $H \defeq \ell^{2}(G)$ via the left regular representation. Then $G$ is skew-well-placed in $\UU(H)_{\mathrm{sot}}$. In particular, $G$ is well-placed in both $\UU(H)_{\mathrm{norm}}$ and $\UU(LG)_{\mathrm{sot}}$. \end{corollary}

\begin{proof} Recall that $\Prob(G) = \Meanu(G_{\Lsh})$ by Remark~\ref{remark:discrete.uniformity}. Due to Theorem~\ref{theorem:skew.exactness.for.unitary.groups}, the subgroup $G$ is skew-well-placed in $\UU(H)_{\mathrm{sot}}$, i.e., there is a $G$-left-equivariant, uniformly continuous map \begin{displaymath}
    \varphi \colon \, (\UU(H)_{\mathrm{sot}})_{\Lsh} \, \longrightarrow \, \Prob(G) .
\end{displaymath} Since the norm topology is finer than the strong operator topology, \begin{displaymath}
    \varphi \colon \, (\UU(H)_{\mathrm{norm}})_{\Lsh} \, \longrightarrow \, \Prob(G)
\end{displaymath} is uniformly continuous, too. Moreover, as $\UU(H)_{\mathrm{norm}}$ is a SIN group, \begin{displaymath}
    \varphi \colon \, (\UU(H)_{\mathrm{norm}})_{\Rsh} \, \longrightarrow \, \Prob(G)
\end{displaymath} is uniformly continuous. Consequently, $G$ is well-placed in $\UU(H)_{\mathrm{norm}}$. Finally, since $\UU(LG)_{\mathrm{sot}}$ is a SIN group as well, the restriction \begin{displaymath}
    \varphi\vert_{\UU(LG)} \colon \, (\UU(LG)_{\mathrm{sot}})_{\Rsh} \, \longrightarrow \, \Prob(G)
\end{displaymath} is uniformly continuous, so that $G$ is well-placed in $\UU(LG)_{\mathrm{sot}}$, too. \end{proof}

\begin{corollary}
    Let $H$ be a Hilbert space. The following are equivalent. \begin{itemize}
        \item[$(1)$] $\UU(H)_{\mathrm{norm}}$ is exact.
        \item[$(2)$] $\UU(H)_{\mathrm{sot}}$ is skew-exact.
        \item[$(3)$] $\dim (H) < \infty$.
    \end{itemize}
\end{corollary}

\begin{proof} (1)$\vee$(2)$\Longrightarrow$(3). Suppose that $\dim (H) = \infty$. Let $X$ be a set such that $H \cong \ell^{2}(X)$. By Gromov's work \cite{Gromov00,Gromov03}, there is a non-exact countable group $\Gamma$. Choose any group $\Gamma'$ such that $\vert \Gamma' \vert = \vert X \vert$. Then the group $G \defeq \Gamma \times \Gamma'$ is non-exact, and $\vert G \vert = \vert X \vert$, thus $H \cong \ell^{2}(G)$. By Corollary~\ref{corollary:skew.exactness.for.unitary.groups} and Theorem~\ref{theorem:amenable.actions.for.well.placed.subgroups} (or Theorem~\ref{theorem:skew.amenable.actions.for.skew.well.placed.subgroups}), it follows that $\UU(H)_{\mathrm{norm}}$ is not exact. Thanks to Corollary~\ref{corollary:skew.exactness.for.unitary.groups} and Corollary~\ref{corollary:skew.exactness}, we see that $\UU(H)_{\mathrm{sot}}$ is not skew-exact.

(3)$\Longrightarrow$(1)$\wedge$(2). If $\dim (H) < \infty$, then $\UU(H)_{\mathrm{norm}} = \UU(H)_{\mathrm{sot}}$ is compact, thus exact for trivial reasons. \end{proof}

\begin{remark}
The same arguments as above apply to any closed subgroup of $\UU(\ell^2(G))_{\mathrm{sot}}$ containing $\lambda(G)$ for some non-exact discrete group $G$.

In particular, by Corollary~\ref{corollary:skew.exactness.for.unitary.groups} and Corollary~\ref{corollary:well-placed.subgroups.inherit}, the unitary group of the group von Neumann algebra $LG$ is non-exact for any non-exact discrete group~$G$. We do not know if the converse holds, i.e., if $\UU(LG)_{\mathrm{sot}}$ is exact for exact $G$. In particular, we were unable to decide if $\UU(LF_2)_{\mathrm{sot}}$ is exact, where $F_2$ is the free group on two generators. There is a similar situation for unitary groups of reduced C$^*$-algebras of discrete groups.
\end{remark}

\subsection{Well-placedness in $L^{0}$ groups}

We exhibit a natural family of well-placed and skew-well-placed embeddings of Polish groups.

Consider a probability space $(X,\mathcal{B},\mu)$ and a separable metrizable topological space $Y$. We let $L^{0}(\mu,Y)$ denote the set of all equivalence classes of Borel measurable functions from $(X,\mathcal{B})$ to $Y$ up to equality $\mu$-almost everywhere. For any compatible metric $d$ on $Y$, there is a metric $d_{\mu}^{0}$ on $L^{0}(\mu,Y)$ defined by \begin{displaymath}
	d_{\mu}^{0}(f,g) \, \defeq \, \inf \{ \varepsilon \in \R_{>0} \mid \mu(\{ x \in X \mid d(f(x),g(x)) > \varepsilon \}) \leq \varepsilon \}
\end{displaymath} for all $f,g \in L^{0}(\mu,Y)$. For any two compatible metrics $d$ and $\delta$ on $Y$, the two induced metrics $d_{\mu}^{0}$ and $\delta^{0}_{\mu}$ generate the same topology on $L^{0}(\mu,Y)$ (see~\cite[Proposition~6]{moore} and \cite[p.~6, Corollary]{moore}). We equip $L^{0}(\mu,Y)$ with the resulting topology, which is uniquely determined by $\mu$ and the topology of $Y$, and which is called the \emph{topology of convergence in measure with respect to $\mu$}. In case $d$ is a complete compatible metric on $Y$, the metric $d_{\mu}^{0}$ is complete as well. Thus, if $Y$ is Polish, then $L^{0}(\mu,Y)$ is completely metrizable. As established by Moore~\cite[Proposition~7]{moore}, if the $\sigma$-algebra $\mathcal{B}$ is countably generated and $Y$ is Polish, then $L^{0}(\mu,Y)$ is Polish. Suppose that $G$ is a second-countable topological group. Then $L^{0}(\mu,G)$, endowed with the pointwise multiplication (of representatives of equivalence classes) and the topology of convergence in measure with respect to $\mu$, constitutes a topological group. It is straightforward to verify that the map \begin{displaymath}
	G \, \longrightarrow \, L^{0}(\mu,G), \quad g \, \longmapsto \, \tilde{g}
\end{displaymath} with \begin{displaymath}
	\tilde{g} \, \equiv_{\mu} \, g \qquad (g \in G)
\end{displaymath} is a topological group embedding, and henceforth we will view $G$ as a topological subgroup of $L^{0}(\mu,G)$ via this embedding. Finally, we note that, if $d$ is a right-invariant (resp., left-invariant) compatible metric on $G$, then the corresponding compatible metric $d_{\mu}^{0}$ will be right-invariant (resp., left-invariant) on $L^{0}(\mu,G)$.

\begin{proposition}\label{proposition:l0} Let $(X,\mathcal{B},\mu)$ be a probability space and let $G$ be a second-countable topological group. The map \begin{displaymath}
	\varphi \colon \, L^{0}(\mu,G) \, \longrightarrow \, \Prob(G), \quad f \, \longmapsto \, f_{\ast}(\mu)
\end{displaymath} is both $G$-left- and $G$-right-equivariant. For any compatible bounded metric $d$ on $G$ and for all $f,g \in L^{0}(\mu,G)$, \begin{displaymath}
	d_{\mathrm{MT}}(\varphi(f),\varphi(g)) \, \leq \, (\diam (G,d) + 1)d_{\mu}^{0}(f,g) .
\end{displaymath} In particular, $G$ is both well-placed and skew-well-placed in $L^{0}(\mu,G)$. \end{proposition}

\begin{proof} It is easy to see that $\varphi$ is well defined and both $G$-left- and $G$-right-equivariant. Now, let $d$ be a compatible bounded metric on $G$ and let $f,g \in L^{0}(\mu,G)$. We show that \begin{displaymath}
	d_{\mathrm{MT}}(\varphi(f),\varphi(g)) \, \leq \, (\diam (G,d) + 1)d_{\mu}^{0}(f,g) ,
\end{displaymath} i.e., \begin{displaymath}
	\forall \varepsilon \in \R_{>0} \colon \quad d_{\mu}^{0}(f,g) < \varepsilon \ \Longrightarrow \ d_{\mathrm{MT}}(\varphi(f),\varphi(g)) \leq (\diam (G,d) + 1)\varepsilon .
\end{displaymath} To this end, let $\varepsilon \in \R_{>0}$ with $d_{\mu}^{0}(f,g) < \varepsilon$. Then $\mu(B) \leq \varepsilon$ for \begin{displaymath}
    B \, \defeq \, \{ x \in X \mid d(f(x),g(x)) > \varepsilon \} .
\end{displaymath} For all $h \in \Lip_{1}(G,d)$, we see that \begin{align*}
	&\left\vert \int h \, \mathrm{d}\varphi(f) - \int h \, \mathrm{d}\varphi(g) \right\vert \, = \, \left\vert \int h \circ f \, \mathrm{d}\mu - \int h \circ g \, \mathrm{d}\mu \right\vert \\
	& \qquad \leq \, \int \left\vert (h \circ f) - (h \circ g) \right\vert \, \mathrm{d}\mu \\
	& \qquad = \, \int_{B} \left\vert (h \circ f) - (h \circ g) \right\vert \, \mathrm{d}\mu + \int_{X \setminus B} \left\vert (h \circ f) - (h \circ g) \right\vert \, \mathrm{d}\mu \\
	&\qquad \leq \, \diam(G,d) \varepsilon + \varepsilon \, = \, (\diam(G,d) + 1) \varepsilon .
\end{align*} Hence, $d_{\mathrm{MT}}(\varphi(f),\varphi(g)) \leq (\diam(G,d) + 1)\varepsilon$ as desired. Applying the established inequality to any right-invariant (resp., left-invariant) compatible bounded metric on $G$ and using $G$-left-equivariance of $\varphi$ along with Remark~\ref{remark:mass.transportation}, we conclude that $G$ is both well-placed and skew-well-placed in $L^{0}(\mu,G)$. \end{proof}

Let us point out a note-worthy consequence of the result above. Let $G$ be a second-countable topological group and let $(X,\mathcal{B},\mu)$ be a probability space. Of course, if $G$ is amenable, then by combining density of the subgroup of ($\mu$-equivalence classes of) $G$-valued $\mathcal{B}$-simple functions on $X$ in $L^{0}(\mu,G)$ with some well-known closure properties of the class of amenable topological groups (see~\cite[4.5--4.8]{MR222208}) we see that $L^{0}(\mu,G)$ is amenable, too. Conversely, if $L^{0}(\mu,G)$ is amenable, then Proposition~\ref{proposition:l0} and Theorem~\ref{theorem:amenable.actions.for.well.placed.subgroups} together imply amenability of $G$. This result was first established in~\cite[Theorem~1.1, (2)$\Rightarrow$(1)]{PestovSchneider}.

\section*{Appendix}
Here, we give a proof of Proposition~\ref{proposition:ad.definition}. 
We need two lemmas. 
\begin{lemma}\label{lem Ch is conti}
Let $h\in C_c(X\times G)_{\ge 0}$. Then the map $C\colon X\ni x\mapsto \int_Gh^xd\nu\in \R$ is continuous. Here, $h^x(g)=h(x,g)$ for $(x,g)\in X\times G$. 
\end{lemma}
\begin{proof} Let $(x_j)_{j\in J}$ be a net in $X$ converging to $x\in X$. 
    Let $K=\overline{{\rm spt}(h)}$, $L={\rm Proj_G}(K)$, where ${\rm Proj}_G\colon X\times G\ni (x,g)\mapsto g\in G$ is the projection map. Then $K,L$ are compact and $C(y)=\int_Lh(y,g)d\nu(g),\,y\in X$.
    Assume first that $(x,g)\notin K$ for any $g\in L$. 
    Then $C(x)=0$. For each $g\in L$, there exists open neighborhoods $V_g'$ of $x$ in $X$ and $W_g'$ of $g$ in $L$ such that $V_g'\times W_g'\subset X\times L\setminus K$. By the compactness of $L$, there exists a finite set $F\subset L$ for which $L=\bigcup_{g\in F}W_g'$. 
    Set $V'=\bigcap_{g\in F}V_g'$, which is an open neighborhood of $x$ in $X$. 
    Since $x_j\to x$, there exists $j_0\in K$ such that $x_j\in V'$ for $j\ge j_0$. 
    Then $h(x_j,g)=0$ for every $j\ge j_0$ and $g\in L$, whence $C(x_j)=0$. This shows that $\lim_{j \to J} C(x_j) = C(x)$.  

Next, assume that $(x,g)\in K$ for some $g\in L$, and let $\varepsilon>0$. 
    Since $h$ is uniformly continuous on $X\times L$, there exists an open neighborhood $U$ of the diagonal $\Delta_{X\times L}$ in $(X\times L)^2$, such that $|h(y',g')-h(y,g)|<\varepsilon$ for every $((y',g'),(y,g))\in U$.  Let $(y,g)\in X\times L$. Then $((y,g),(y,g))\in \Delta_{X\times L}$, and there exist open neighborhoods $V_{y,g}$ of $y$ in $X$, $W_{y,g}$ of $g$ in $L$, such that $(V_{y,g}\times W_{y,g})^2\subset U$ holds. By compactness, there exists a finite set $F=\{z_1,\dots,z_m\}\subset K\times L$, $z_m=(y_m,g_m)$, such that $X\times L=\bigcup_{k=1}^mV_{z_k}\times W_{z_k}$. Let 
\[I=\{k\in \{1,\dots,m\}\mid \exists g\in L\,{\rm s.t.}\,(x,g)\in V_{z_k}\times W_{z_k}\}\]
Then $I\neq \emptyset$ by assumption. Since $x_j\to x$, there exists $j_0\in J$ such that for every $j\ge j_0$ and for every $k\in I$, $x_j\in V_{z_k}$ holds. 
Now let $g\in L$ be arbitrary. Then 
$(x,g)\in V_{z_k}\times W_{z_k}$ for some $k\in I$. Then we have 
\[((x_j,g),(x,g))\in (V_{z_k}\times W_{z_k})^2\subset U,\]
whence 
\[|h(x_j,g)-h(x,g)|<\varepsilon.\]
Since $g\in L$ is arbitrary, we obtain that $|C(x_j)-C(x)|\le \varepsilon \nu(L)$ for all $j \in J$ with $j \geq j_0$. This shows that $\lim_{j \to J} C(x_j)=C(x)$. Therefore, $C$ is continuous. \end{proof}

\begin{lemma}\label{lem weakstarandUEB}
Let $G$ be a locally compact group, $X$ be a compact space. 
Let $\mu\colon X\to \Prob(G,\nu)$ be a map such that 
$\bigcup_{x\in X}\spt(\mu(x))$ is relatively compact. 
Then the following two conditions are equivalent. 
\begin{list}{}{}
\item[{\rm (i)}] $\mu\colon X\to \Prob(G,\nu)$ is $\sigma(C_0(G)^*,C_0(G))$-continuous. 
\item[{\rm (ii)}] $\Xi_{\nu}^{\Lsh}(\mu)\colon X\ni x\mapsto \Xi_{\nu}^{\Lsh}(\mu(x))\in \Mean_u(G_{\Lsh})$ is {\rm UEB}-continuous. 
\end{list}
\end{lemma}

\begin{proof} (ii)$\implies$(i) follows from the inclusion $C_0(G)\subset \UCB(G_{\Lsh})$, the equality $\ip{\Xi_{\nu}^{\Lsh}(\mu(x))}{f}=\ip{\mu(x)}{f}$ for $f\in C_0(G)$ and the fact that the UEB-topology and the $\sigma(\UCB(G_{\Lsh})^*,\UCB(G_{\Lsh}))$-topology coincide on $\Mean_u(G_{\Lsh})$.

(i)$\implies$(ii) Let $K=\overline{\bigcup_{x\in X}\spt(\mu(x))}$, which is a compact subset of $G$. By Urysohn's lemma, we can choose $\chi\in C_c(G)^+$, $0\le \chi\le 1$, such that $\chi|_K=1$. Then for $f\in \UCB(G_{\Lsh})$, we have 
\eqa{
    \ip{\Xi_{\nu}^{\Lsh}(\mu(x))}{f}&=\int_K\mu(x)(g)f(g)d\nu(g)\\
    &=\int_G\mu(x)(g)\chi(g)f(g)d\nu(g)\\
    &=\ip{\mu(x)}{\chi\cdot f},
}
because $\chi\cdot f\in C_0(G)$. This implies that $x\mapsto \ip{\Xi_{\nu}^{\Lsh}(\mu(x))}{f}$ is continuous, which implies (ii). \end{proof}

\begin{proof}[Proof of Proposition~\ref{proposition:ad.definition}] 
Suppose that the action is AD-amenable. Then by \cite[Proposition 2.2]{AD23amenabilityexactness}, there exists a net $(h_{\iota})_{\iota\in I}$ in $C_c(X\times G)_{\ge 0}$ such that 
\begin{list}{}{}
\item[(a)] $\lim_{\iota \to I }\int_G h_{\iota}^xd\nu=1$ uniformly in $x\in X$. 
\item[(b)] $\lim_{\iota \to I}\sup_{x\in X}\|h_{\iota}^{gx}-gh_{\iota}^x\|_{\nu,1}=0$ compact uniformly in $g\in G$. 
\end{list}
Here, $h_{\iota}^x\in \Prob(G,\nu)$ is defined by $h_{\iota}^x(g)=h_{\iota}(x,g),\, x\in X,\,g\in G$. 
By (a) we may assume that $C_{\iota}^x=\int_Gh_{\iota}^xd\nu>0$ for all $\iota,x$. 

For each $x\in X$ define $\mu_{\iota}(x)\in C_c(G)_{\ge 0}$ by the formula
\[\mu_{\iota}(x)(g)=\frac{h_{\iota}^x(g)}{\int_Gh_{\iota}^xd\nu},\,\,g\in G.\]
Then by Lemma \ref{lem weakstarandUEB}, 
the map $X\ni x\mapsto \Xi_{\nu}^{\Lsh}(\mu_{\iota}(x))\in \Mean_u(G_{\Lsh})$ is UEB-continuous. By the same argument, the map $X\ni x\mapsto \Xi_{\nu}^{\Rsh}(\mu_{\iota}(x))\in \Mean_u(G_{\Rsh})$ is also UEB-continuous.
\end{proof}

\section*{Acknowledgments}
H.A.\ is supported by Japan Society for the Promotion of Sciences (JSPS) KAKENHI 20K03647. A.T.\ thanks Nicolas Monod for inspiring discussions and sharing of Definition \ref{def:amenemb} at the start of the project. Moreover, the authors express their sincere gratitude towards the anonymous referee, in particular for suggesting Question~\ref{question:extension}(1).

\bibliographystyle{plain}

\end{document}